\documentclass[a4paper,11pt]{amsart}
\usepackage[T1,T2A]{fontenc}
\usepackage[english]{babel}
\usepackage{amsmath}
\usepackage{amsfonts}
\usepackage{amssymb}
\usepackage{mathrsfs}
\usepackage{amsthm}
\usepackage{hyperref}
\usepackage{graphicx}
\graphicspath{{./Pictures/}}
\usepackage{color}
\usepackage{enumerate}
\usepackage{textcase}

\newcommand{\df}{\buildrel\mathrm{def}\over=}
\newcommand{\Bell}{\boldsymbol{B}}
\newcommand{\RR}{\mathrm R}
\newcommand{\LL}{\mathrm L}
\newcommand{\eps}{\varepsilon}
\newcommand{\vf}{\varphi}
\newcommand{\vk}{\varkappa}
\newcommand{\R}{\mathbb R}

\newcommand{\Adm}{\mathrm{Adm}}
\newcommand{\const}{\mathrm{const}}

\newcommand{\cF}{\mathcal F}
\newcommand{\bE}{\mathbb{E}}
\newcommand{\bP}{\mathbb{P}}

\newcommand{\half}{\tfrac12}
\newcommand{\eq}[2]{\begin{equation}\label{#1} #2 \end{equation}}

\newcommand{\scalp}[2]{}

\newtheorem{Le}{Lemma}[section]
\newtheorem{Prop}[Le]{Proposition}
\newtheorem{Th}[Le]{Theorem}
\newtheorem{Cor}[Le]{Corollary}
\theoremstyle{definition}
\newtheorem{Def}[Le]{Definition}
\newtheorem{Rem}[Le]{Remark}
\numberwithin{equation}{section}

\newcommand{\ti}[1]{_{\scriptscriptstyle\text{\rm #1}}}
\newcommand{\ut}[1]{^{\scriptscriptstyle\text{\rm #1}}}

\newcommand{\Oe}{\Omega_\eps}
\newcommand{\Om}[1]{\Omega\ut{#1}_\eps}

\newcommand\mytitle{Some extremal problems for martingale transforms, I}
\newcommand\mytitleshort{Extremal problems for martingale transforms}

\begin{document}
\title[\mytitleshort]{\mytitle}

\author{V.~I.~Vasyunin, P.~B.~Zatitskii}


\begin{abstract}
With this paper, we begin a series of studies of extremal problems for estimating distributions of
martingale transforms of bounded martingales. The Bellman functions corresponding to such problems
are pointwise minimal diagonally concave functions on a horizontal strip, satisfying certain given boundary conditions. We describe the basic structures that arise when constructing such functions and present a solution in the case of asymmetric boundary conditions and a sufficiently small width of the strip.
\end{abstract}

\thanks{This work is supported by the Russian Science Foundation, grant number 19-71-10023.}
\maketitle

\section{Introduction}
With this article, we start a series of papers in which we are going to describe the algorithm for constructing Bellman functions for extremal problems described below. We are inspired by the successful description of the algorithm for constructing Bellman functions for a sufficiently wide class of integral extremal problems on BMO that is realized in~\cite{IOSVZ2016,ISVZ2018}, and also by a recent generalization of these results to more general classes of functions, including among others Muckenhoupt classes, see~\cite{ISVZ2023}. The original extremal problem on an infinite-dimensional functional space is reduced to the problem of constructing the (pointwise) minimal locally concave function in the corresponding two-dimensional domain (parabolic strip in the case of BMO) with a given boundary condition, see~\cite{SZ2016}. The minimal locally concave function turns out to be linear along some direction at each interior point of its domain. Thus, this two-dimensional domain is foliated by linear segments or two-dimensional subdomains, and we call such a structure a foliation. If the foliation is constructed, it is possible to reconstruct the Bellman function and obtain sharp estimates for the original extremal problem.

A similar approach is used to obtain sharp estimates in various questions of martingale theory (and for more general random processes). It goes back to~\cite{Burkholder1984} and is called the Burkholder method. Many examples of usage of this method and a literature review on the topic can be found in~\cite{Adam}. When studying estimates of martingale transformations, instead of locally concave functions the so-called diagonally concave functions arise --- those are concave along line segments of the form $x_1\pm x_2=\text{const}$\footnote{In Burkholder's papers, such functions are called zig-zag concave.}.

Recently, it was discovered that a certain class of extremal problems for martingale transforms is closely related to corresponding extremal problems on BMO. The Burkholder method reduces the problem of estimating the distribution of a martingale transform of a bounded martingale to finding the minimal diagonally concave function on a horizontal strip, satisfying a certain symmetric boundary condition on the upper and lower boundaries of the strip. It turns out that the structure of such a function in some sense coincides with the structure of the minimal locally concave function on a parabolic strip with the corresponding boundary data; details can be found in~\cite{SymStr}. After discovering this connection, a desire to consider a broader class of problems on a~horizontal strip arose naturally, where arbitrary, not necessarily symmetric, boundary values are specified on the boundaries of the strip. When solving this problem, we were guided by two considerations. First, such a generalization seemed to us the most natural and was dictated by the logic of the development of the theory. That is, we are interested in understanding how to construct the Bellman function not for some specific problem but to understand the algorithm that allows constructing the necessary function for a wide class of problems. Since we were guided by internal tasks of method development, we were not looking for specific applications. Although there will be many examples in this series of papers, all of them are examples of applying the theory but are not external problems. Second, since there is considerable interest in various estimates of martingale transforms (see~\cite{Adam}), we do not lose hope that our theory may find application in the future for some external problems. However, we ask the reader not to seek such applications in the proposed text.

\bigskip

Let's proceed with the formal statement of the problem. Let $(S, \cF, \bP)$ be a standard probability space, and $(\cF_n)_{n=0}^\infty$ be a dense filtration in $\cF$ with $\cF_0$ being a trivial algebra. Let $(\vf_n, \cF_n)_{n=0}^\infty$ and $(\psi_n, \cF_n)_{n=0}^\infty$ be two martingales on this space generated by integrable functions $\vf, \psi \in L^1(S, \cF, \bP)$:
$$
\vf_n = \bE(\vf\mid \cF_n),\qquad 
\psi_n = \bE(\psi\mid \cF_n).
$$
Let $d\vf_n = \vf_{n}-\vf_{n-1}$ and $d\psi_n = \psi_{n}-\psi_{n-1}$ for $n\geq 1$ be their martingale differences. Due to the triviality of $\cF_0$, we have $\vf_0=\bE\vf$, $\psi_0=\bE\psi$. We will say that the martingale $(\psi_n)$ is a martingale transformation of the martingale $(\vf_n)$ if there exists a sequence of functions $(\alpha_n)_{n=1}^\infty$ on $(S, \cF, \bP)$ such that  $\alpha_n$ is measurable with respect to $\cF_{n-1}$ and $d\psi_n = \alpha_n d\vf_n$ for each $n\geq 1$. In what follows, we restrict ourselves to the case when $|\alpha_n|=1$ almost everywhere.

\begin{Def}
For a given $\eps>0$, consider the following two-dimensional domain
\eq{170502}{
\Oe\df\{x=(x_1,x_2)\colon -\infty<x_1<\infty,\ -\eps\le x_2\le\eps\}.
}
A pair of functions $(\vf,\psi)$ generating martingales $(\vf_n)$ and $(\psi_n)$ is called admissible for a point $x\in\Oe$ if $\vf_0=x_1$, $\psi_0 = x_2$, and $|\psi|=\eps$ almost everywhere. The set of all admissible pairs for $x$ is denoted by $\Adm_\eps(x)$.
\end{Def}

\begin{Def}
Let $f_\pm\colon \mathbb{R}\to\mathbb{R}$ be two measurable functions. Set $f(x_1,\pm\eps)=f_\pm(x_1)$, $x_1 \in \mathbb{R}$. Define the function $\Bell$ on $\Oe$ by the formula
\eq{170503}{
\Bell(x) = \sup\Big\{\bE f(\vf, \psi) \colon (\vf,\psi) \in \Adm_\eps(x)\Big\},\qquad x \in \Oe.
}
\end{Def}

The main goal of this study is to construct the Bellman function $\Bell$ for given $f_\pm$. By solving this problem, we provide a description of the joint distribution of the function $\vf$ and its martingale transformation $\psi$. In a recent work~\cite{SymStr}, this problem is solved for the case of symmetric boundary conditions: a description of $\Bell$ is given under the condition $f_+=f_-$.

\begin{Def}
A function defined on a two-dimensional domain is called diagonally concave if it is concave in the directions $x_2\pm x_1=\text{const}$.
\end{Def}

The following theorem goes back to D. Burkholder (the proof can be found, for example, in Chapter 2 of the book~\cite{Adam}).
\begin{Th}\label{th170901}
The function $\Bell$ is diagonally concave on the strip $\Omega_\eps$ and satisfies the boundary conditions
\eq{050501}{
\Bell(x_1,\pm\eps)=f_\pm(x_1), \qquad x_1 \in \mathbb{R}.
}
Moreover\textup, it is the pointwise minimal among all functions satisfying these properties.
\end{Th}

\begin{Def}
A function $B$ defined on a subdomain $\Omega$ of $\Oe$ is called a Bellman candidate (a~candidate for the Bellman function) on $\Omega$ if the following conditions are satisfied:
\begin{enumerate}
\item $B$ is diagonally concave on $\Omega$;
\item $B$ satisfies the boundary conditions~\eqref{050501} on $\Omega \cap \partial \Oe$;
\item for any point $y$ in $\Omega$, one of the following conditions holds:
\begin{itemize}
\item
there exists a straight line segment going in one of the directions $x_1\pm x_2=\text{const}$, that connects $y$ with some point on $\partial \Oe$, where the function $B$ is linear;
\item 
the function $B$ is linear in both directions $x_1\pm x_2=\text{const}$ in a neighborhood  of $y$.
\end{itemize}
\end{enumerate}
\end{Def}

In a recent paper~\cite{Misha}, M. I. Novikov proved a remarkable theorem that allows one to verify that a given diagonally concave function is pointwise minimal among all such functions with the same boundary values. We present its simplified version, that is Theorem 1.5.3 from~\cite{SymStr}.

\begin{Def}
Let $G\colon \Oe\to \R$ be a diagonally concave function. We say that $G$ is extremal in the direction $(1,\pm 1)$ at $x$ if $x$ is the endpoint of some segment $\ell$, the other endpoint of $\ell$ lies on $\partial\Oe$, $\ell$ is parallel to $(1,\pm 1)$, the function $G|_\ell$ is linear, and $G$ is differentiable in the direction $(1,\pm 1)$ at~$x$.
\end{Def}

\begin{Th}[\cite{Misha}]\label{MishasTh}
Let a function $G\colon \Oe\to \mathbb{R}$ be diagonally concave and upper semi-continuous with a discrete set of discontinuities. Let $|G(x)| \lesssim e^{|x_1|/{\tilde\eps}},$ $x \in \Oe,$ for some $\tilde\eps > \eps$. Assume that for each interior point $x\in \Oe$ one of the following conditions holds\textup:
\begin{enumerate} 
\item[\textup{1)}] $G$ is linear and extremal in the direction $(1,1)$ at $x$\textup;
\item[\textup{2)}] $G$ is linear and extremal in the direction $(1,-1)$ at $x$\textup;
\item[\textup{3)}] $G$ is extremal in both directions $(1,\pm 1)$ at $x$\textup;
\item[\textup{4)}] $G$ is linear in both directions $(1,\pm 1)$ in a neighborhood of $x$.
\end{enumerate}
Then $G$ is pointwise minimal among all diagonally concave functions with the same boundary values.
\end{Th}

In the future, we will construct diagonally concave functions on $\Oe$ that satisfy the assumptions of Theorem~\ref{MishasTh}, thus we obtain minimal functions. According to Theorem~\ref{th170901}, these functions will be the Bellman functions~\eqref{170503} for the original extremal problems.

For simplicity, we assume that the boundary functions $f_\pm$ are smooth, although for most of the presented constructions, $C^3$-smoothness is sufficient. To use Theorem~\ref{MishasTh}, we also assume that the functions $f_\pm$ satisfy the estimate $|f_\pm(x_1)|\leq C e^{|x_1|/\tilde\eps}$ for some $\tilde\eps>\eps$.

\section{Simplest Foliation}
\label{Sec2}

Since we are looking for the minimal possible diagonally concave function, its concavity must be degenerate in each point in at least one direction: either along the line segment $x_2-x_1=\text{const}$ or along the line segment $x_2+x_1=\text{const}$. We call such line segments {\it extremal};
 in the first case, we call them right extremal segments, and in the second case, left extremal segments.

In this section, we consider the simplest case when a subregion of $\Oe$ is foliated only by one type of extremal segments: either only right ones or only left ones.

\begin{Def}
The subregion of $\Oe$ foliated by right extremal segments $x_1-x_2=u$, where $u_1\le u\le u_2$, is denoted by $\Om{R}(u_1,u_2)$. Similarly, the subregion of $\Oe$ foliated by left extremal segments $x_1+x_2=u$, where $u_1\le u\le u_2$, is denoted by $\Om{L}(u_1,u_2)$.
\end{Def}

The cases of right and left extremal segments are symmetric. Suppose that $B$ is linear along right extremal segments which foliate $\Om{R}(u_1,u_2)$. This means that for any $u$, $u_1\le u\le u_2$, the boundary points $(u-\eps,-\eps)$ and $(u+\eps,\eps)$ are connected by an extremal segment along which $B$ is linear. In other words,
$$
B(u+s,s)=\frac{\eps+s}{2\eps}f_+(u+\eps)+\frac{\eps-s}{2\eps}f_-(u-\eps)\,,\quad|s|\le\eps\,,
$$
or equivalently,
\eq{180501}{
B(x_1,x_2)=\frac{\eps+x_2}{2\eps}f_+(x_1-x_2+\eps)+\frac{\eps-x_2}{2\eps}f_-(x_1-x_2-\eps)\,.
}

By construction, $B$ is linear (and therefore concave) along the direction $x_2-x_1=\text{const}$, and we need to check its concavity along the direction $x_2+x_1=\text{const}$. For a fixed point $(x_1,x_2)$ define
\begin{align*}
h(s)&=B(x_1-s,x_2+s)
\\
&=\frac{\eps+x_2+s}{2\eps}f_+(x_1-x_2-2s+\eps)+\frac{\eps-x_2-s}{2\eps}f_-(x_1-x_2-2s-\eps)
\end{align*}
and check when this function is concave at $s=0$. Compute the derivatives:
\begin{align*}
h'(s)&=\frac1{2\eps}(f_+-f_-)-\frac{\eps+x_2+s}\eps f_+'-\frac{\eps-x_2-s}\eps f_-'\,,
\\
h''(s)&=-\frac2\eps(f_+'-f_-')+2\frac{\eps+x_2+s}\eps f_+''+2\frac{\eps-x_2-s}\eps f_-''\,.
\end{align*}
Here the values of the function $f_+$ and its derivatives are evaluated at $x_1-x_2-2s+\eps$, and the values of $f_-$ and its derivatives are evaluated at  $x_1-x_2-2s-\eps$.

Fix an extremal segment passing through the point $(u,0)$, that is, the line $x_1-x_2=u$. Concavity of $B$ at all points of this extremal segment in the orthogonal direction implies that the inequality
$$
\half h''(0)=-\frac1\eps(f_+'(u+\eps)-f_-'(u-\eps))
+\frac{\eps+x_2}\eps f_+''(u+\eps)+\frac{\eps-x_2}\eps f_-''(u-\eps)\le0\,
$$
holds for all $x_2 \in [-\eps,\eps]$. This expression is linear in $x_2$, so it is non-negative for all $x_2\in [-\eps,\eps]$ if and only if it is non-negative at $x_2=\pm\eps$. Thus, we obtain the conditions
\eq{190501}{
\begin{aligned}
f_+'(u+\eps)-f_-'(u-\eps)-2\eps f_+''(u+\eps)&\ge0\,;
\\
f_+'(u+\eps)-f_-'(u-\eps)-2\eps f_-''(u-\eps)&\ge0\,.
\end{aligned}
}

We proved the following statement.
\begin{Prop}
\label{200505}
The function $B$ defined by~\eqref{180501} is a Bellman candidate in the region $\Om{R}(u_1,u_2)$ if and only if conditions~\eqref{190501} are satisfied for all $u,$ $u\in(u_1,u_2)$.
\end{Prop}

In the symmetric case of left extremal segments, the Bellman candidate is given by the formula
\eq{200501}{
B(x_1,x_2)=\frac{\eps+x_2}{2\eps}f_+(x_1+x_2-\eps)+\frac{\eps-x_2}{2\eps}f_-(x_1+x_2+\eps)\,.
}
Instead of conditions~\eqref{190501}, we obtain the following inequalities:
\eq{200502}{
\begin{aligned}
f_-'(u+\eps)-f_+'(u-\eps)-2\eps f_+''(u-\eps)&\ge0\,;
\\
f_-'(u+\eps)-f_+'(u-\eps)-2\eps f_-''(u+\eps)&\ge0\,.
\end{aligned}
}
We formulate a symmetric statement characterizing a Bellman candidate in this case.
\begin{Prop}
\label{200506}
The function $B$ defined by~\eqref{200501} is a Bellman candidate in the region $\Om{L}(u_1,u_2)$ if and only if conditions~\eqref{200502} are satisfied for all $u,$ $u\in(u_1,u_2)$.
\end{Prop}

As $\eps$ approaches $0$, both inequalities in condition~\eqref{190501} take the same limiting form
\eq{200503}{
f_+'(u)\ge f_-'(u).
}
The limiting form of condition~\eqref{200502} (as $\eps$ approaches $0$) is the opposite inequality
\eq{200504}{
f_+'(u)\le f_-'(u).
}

This leads to the following statement.

\begin{Prop}
Let $u \in \mathbb{R}$. If $f_+'(u) > f_-'(u),$ then for sufficiently small~$\eps$ there exists a $\delta,$ $\delta = \delta(\eps),$ such that the function $B$ defined by~\eqref{180501} is a Bellman candidate in $\Om{R}(u-\delta,u+\delta)$. If~$f_+'(u) < f_-'(u),$ then for sufficiently small~$\eps$ there exists a $\delta,$ $\delta = \delta(\eps),$ such that the function $B$ defined by~\eqref{200501} is a Bellman candidate in $\Om{L}(u-\delta,u+\delta)$.
\end{Prop}

\begin{proof}
Since the functions $f_\pm$ are sufficiently smooth, the condition $f_+'(u) > f_-'(u)$ implies that conditions~\eqref{190501} hold for sufficiently small $\eps$ not only for a given $u$ but also in a neighborhood of size $O(\eps)$. Therefore, the statement follows from Proposition~\ref{200505}. In the case of $f_+'(u) < f_-'(u),$ the conclusion follows from Proposition~\ref{200506}.
\end{proof}

\begin{Prop}
\label{061101}
Let the functions $f_\pm'''$ be uniformly bounded on the entire real line. If $f_+'-f_-'$ is uniformly separated from zero\textup, then for sufficiently small~$\eps$ the function $\Bell$ has a simple right foliation $\Om{R}(-\infty, +\infty)$ in the case of $f_+'-f_-'>0$ and a simple left foliation $\Om{L}(-\infty, +\infty)$ in the case of $f_+'-f_-'<0$.
\end{Prop}

\begin{proof}
If $f_+'-f_-'>0,$ then the conclusion follows directly if we rewrite conditions~\eqref{190501} using Taylor's formula:
$$
\begin{aligned}
f_+'(u-\eps)-f_-'(u-\eps)-2\eps^2 f_+'''(v)&\ge0\quad \text{for some } v \in [u-\eps,u+\eps]\,;
\\
f_+'(u+\eps)-f_-'(u+\eps)+2\eps^2 f_-'''(v)&\ge0\quad \text{for some } v \in [u-\eps,u+\eps]\,.
\end{aligned}
$$
The case $f_+'-f_-'<0$ is similar.
\end{proof}

\section{Examples. Linear and Quadratic Functions}
\label{230201}

Let $f_\pm(t)=a_1^\pm t+a_0^\pm$. The functions $f_\pm'$ are identically equal to the constants $a_1^\pm$, so both expressions on the left side of~\eqref{190501} coincide with $a_1^+-a_1^-$. If $a_1^+ > a_1^-$, then the Bellman function is given by~\eqref{180501} and has a simple foliation $\Om{R}(-\infty, +\infty)$ for all $\eps$. Similarly, if $a_1^+ < a_1^-$, then both expressions on the left side of~\eqref{200502} coincide with $a_1^--a_1^+$, so the Bellman function is given by~\eqref{200501} and has a simple foliation $\Om{L}(-\infty, +\infty)$. If $a_1^+=a_1^-\equiv a_1$, then the Bellman function is~linear:
$$
\Bell(x_1,x_2)=a_1x_1+\frac{a_0^+-a_0^-}{2\eps}x_2+\frac{a_0^++a_0^-}2\,.
$$

Now let $f_\pm(t)=a_2^\pm t^2+a_1^\pm t+a_0^\pm$. Suppose $a_2^-<a_2^+$ and define $u_0=\frac{a_1^--a_1^+}{2(a_2^+-a_2^-)}$. Then condition~\eqref{190501} is satisfied for $u\ge u_0+\eps$, and condition~\eqref{200502} is satisfied for $u\le u_0-\eps$. Thus, we have two regions $\Om{R}(u_0+\eps,+\infty)$ and $\Om{L}(-\infty,u_0-\eps)$ with simple foliations, and the remaining triangle with vertices $(u_0-2\eps,\eps)$, $(u_0,-\eps)$, $(u_0+2\eps,\eps)$. In this triangle, the Bellman function is linear in both diagonal directions:
$$
\Bell(x_1,x_2)=a_2^+(x_1^2-x_2^2+\eps^2)+a_1^+x_1+
\Big(\frac{a_0^+-a_0^-}{2\eps}-\frac{(a_1^+-a_1^-)^2}{8\eps(a_2^+-a_2^-)}\Big)(x_2-\eps)+a_0^+\,.
$$

If $a_2^->a_2^+$, then the foliation is symmetric to the one described above with respect to the $x_1$-axis (we need to swap $f_+$ and $f_-$). So, we are only left to consider the case $a_2^+=a_2^-$. If~$a_1^+\geq a_1^-$, then the relations~\eqref{190501} hold for all $\eps$, and there is a foliation $\Om{R}(-\infty, +\infty)$. Similarly, if $a_1^+\leq a_1^-$, then the relations~\eqref{200502} hold, and there is a foliation $\Om{L}(-\infty, +\infty)$. In particular, if $a_1^+=a_1^-$, then the function is linear along both diagonal directions:
$$
\Bell(x_1,x_2)=a_2(x_1^2-x_2^2+\eps^2)+a_1x_1+\frac{a_0^+-a_0^-}{2\eps}x_2+\frac{a_0^++a_0^-}2\,.
$$

\section{Horizontal Herringbones}
\label{Sec4}

Now we consider foliations that can arise in a neighborhood of a point where 
\eq{210501}{
f_+'(u)=f_-'(u).
}

If two families of extremal segments of distinct direction meet at points lying on some curve, we call such a foliation a \textit{herringbone}, considering this curve as the \textit{spine} of the herringbone and the extremal segments expanding from the points of the spine as its \textit{ribs}. If the ribs go from the spine to the same boundary, we call such a herringbone \textit{vertical} (see Fig.~\ref{fig:vf}), and if they go to opposite boundaries of the strip, we call it \textit{horizontal}.
A vertical herringbone can extend upwards or downwards. 
\begin{figure}[h]
\includegraphics[width=0.45\textwidth]{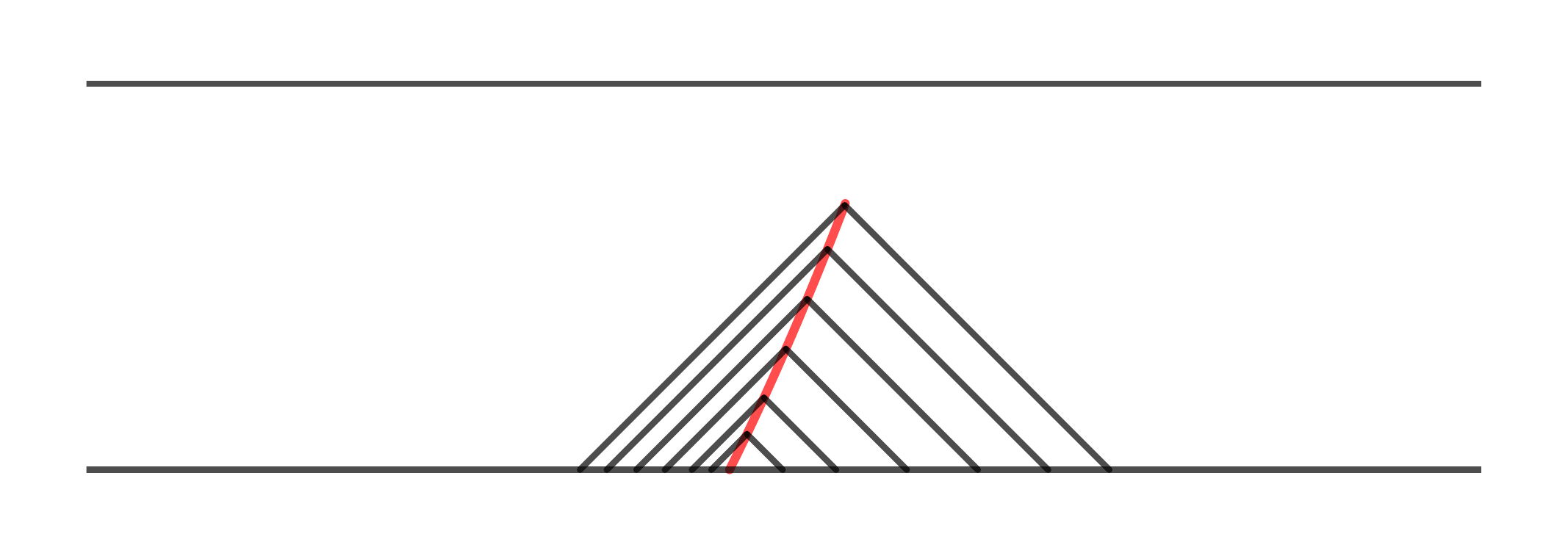}
   \includegraphics[angle=180,origin=c, width=0.45\textwidth]{Vertical_fir.pdf}
   \caption{Vertical herringbones extending from bottom to top and from top to bottom}
    \label{fig:vf}
\end{figure}
A~horizontal herringbone can extend from left to right or from right to left, and depending on the direction of extending, we call it \textit{left} or \textit{right}, respectively (see~Fig.~\ref{fig:hf}).
 \begin{figure}[h]
   \includegraphics[width=0.45\textwidth]{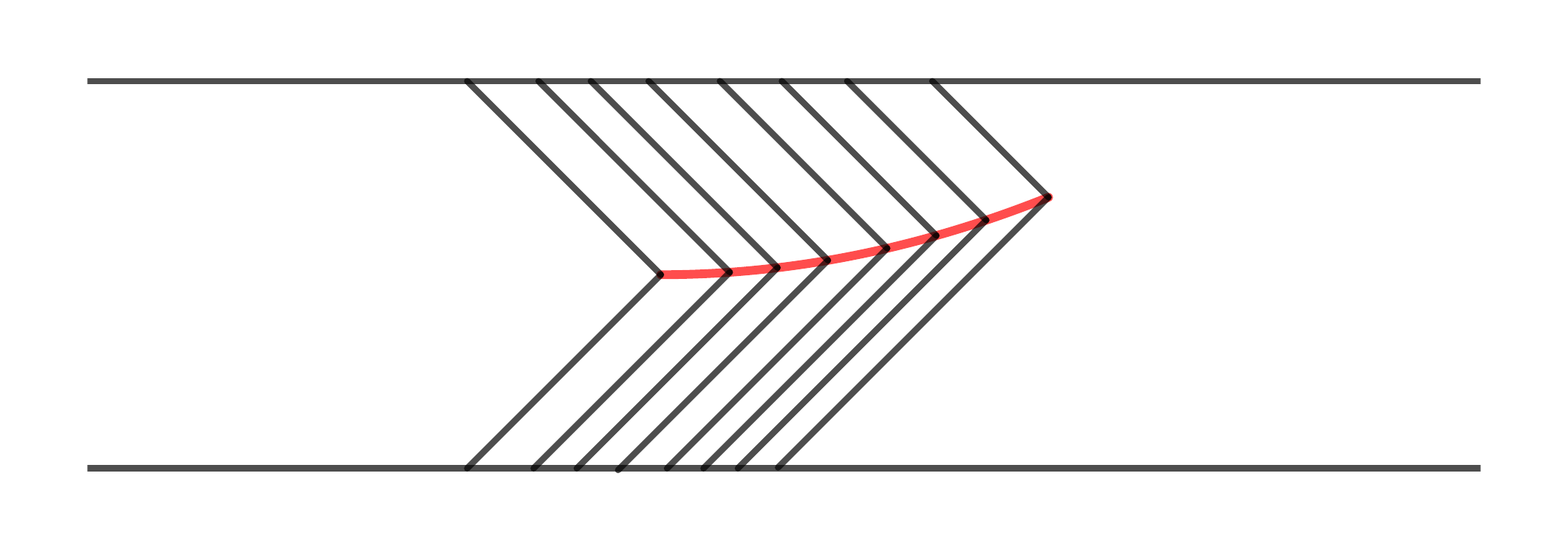}
   \includegraphics[angle=180,origin=c, width=0.45\textwidth] {Horizontal_fir.pdf}
   \caption{Horizontal herringbones, left and right}
   \label{fig:hf}
\end{figure}
A~horizontal herringbone whose spine connects two opposite boundaries of the strip is called a~\textit{fissure}. A fissure can have one of four directions. If the right fissure starts 
from the bottom boundary, we call it \textit{southeast} (SE-fissure); if it starts 
from the top boundary, we call it \textit{northeast} (NE-fissure). If the left fissure 
starts from the bottom boundary, we call it \textit{southwest} (SW-fissure); if it 
starts from the top boundary, we call it 
\textit{northwest} (NW-fissure), see Fig.~\ref{fig:fis}. Note that in all cases, the name of the fissure indicates where it extends from.
\begin{figure}[h]
   \includegraphics[angle=180,origin=c, width=0.23\textwidth]{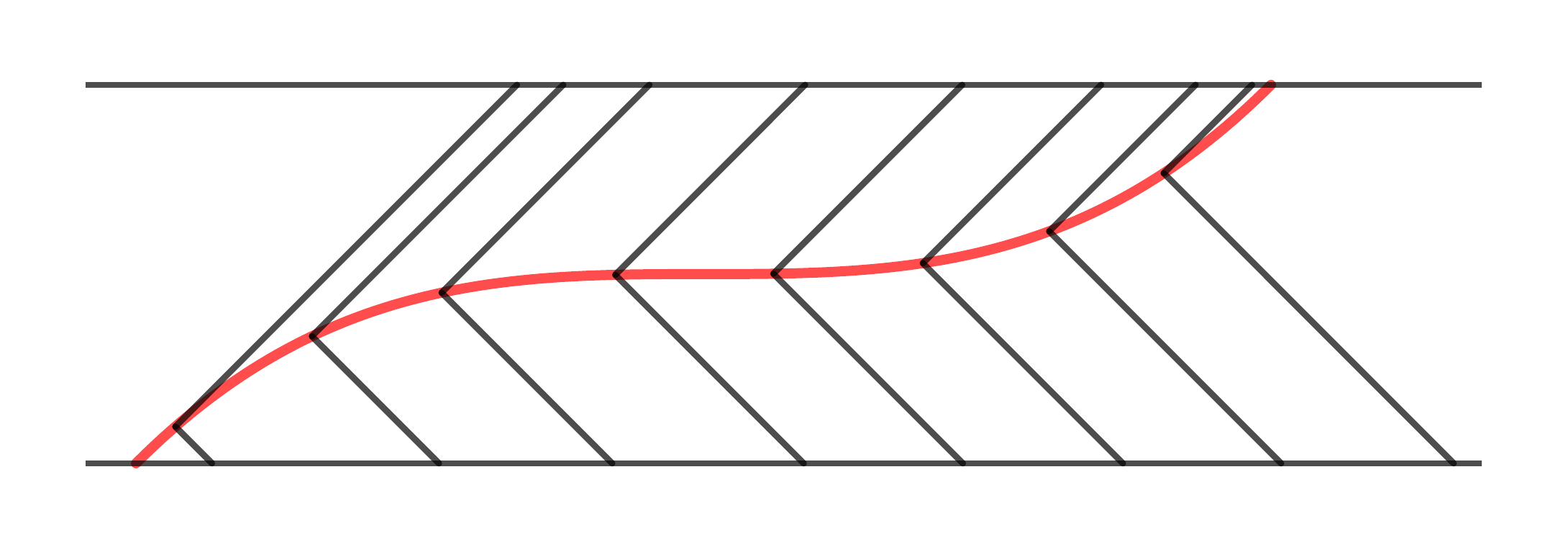}
   \includegraphics[angle=180,origin=c, width=0.23\textwidth]{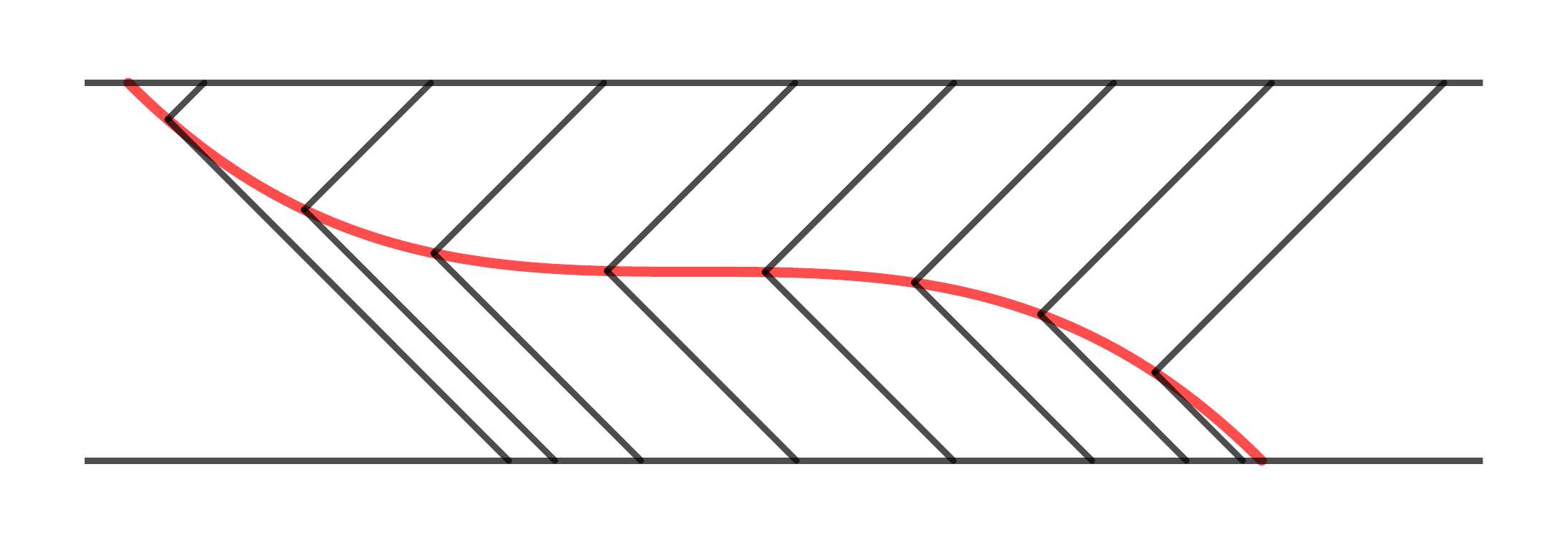}
  \includegraphics[ width=0.23\textwidth]{SW_fissure.pdf}
   \includegraphics[ width=0.23\textwidth]{NW_fissure.pdf}
   \caption{Fissures: SW, NW, NE, SE}
   \label{fig:fis}
\end{figure}

In section~\ref{Fissures}, we will consider different fissures as boundaries between subdomains of left and right foliations. Here we study local behavior of horizontal herringbones. We will only consider the case of a left herringbone (extends from left to right), and for the right one, everything will be symmetric.

Assume that the spine of a left herringbone is the graph of a $C^1$ function $T$ on some interval, say $[u_1,u_2]$, and the foliation is defined as follows. From each point $(u,T(u))$, $u\in[u_1,u_2]$, i.\,e., a~point on the spine, two ribs start: one goes to the point $(u+T(u)-\eps,\eps)$ on the upper boundary, and the other goes to the point $(u-T(u)-\eps,-\eps)$ on the lower boundary, see the left part of Fig.~\ref{fig:hf}. The parameter $u$ itself can be considered as a function of the variable point $x$ lying in the domain foliated by the ribs. The ribs going to the lower boundary foliate the domain between the extremal segments $x_2=x_1-u_1+T(u_1)$ and $x_2=x_1-u_2+T(u_2)$ below the spine of the herringbone. We call them lower ribs. Here the function $x\mapsto u(x)$ is implicitly defined by the identity $x_1-x_2=u-T(u)$, and we have $x_2\leq T(u)$. Similarly, the extremal segments going from the spine to the upper boundary are called upper ribs. Here the identity $x_1+x_2=u+T(u)$ holds, and we have $x_2\geq T(u)$. The condition $|T'|\le 1$ is necessary for such a foliation. We assume the strict inequality $|T'|<1$, except possibly at a finite number of points on the spine.

If we know the values of $B$ on the spine {\bf(}denote it by $A$: $A\!=\!A(u)\! \df\! B(u, T(u))${\bf)}, then we can calculate its values on the ribs using linearity. Before writing the formula for $B$, let us agree on the following simplification of notations: in the future, we will omit the arguments of functions, assuming that $A$, $T$, and their derivatives are always calculated at $u$, the values of $f_-$ and its derivatives are calculated at $u-T-\eps=x_1-x_2-\eps$, and the values of $f_+$ and its derivatives are calculated at $u+T-\eps=x_1+x_2-\eps$. In all special cases where the arguments do not satisfy this rule, they will be explicitly written.

Let us write a formula for a Bellman candidate on the left herringbone:
\eq{220501}{
B(x)=
\begin{cases}
\displaystyle
\frac{\eps-x_2}{\eps-T}A+\frac{x_2-T}{\eps-T}f_+,
& x_2 \geq T(u);
\\
\displaystyle
\frac{\eps+x_2}{\eps+T}A+\frac{T-x_2}{\eps+T}f_-, \rule{0pt}{25pt}
& x_2 \leq T(u).
\end{cases}
}

The functions $A$ and $T$ must be completely determined by the boundary conditions $f_\pm$, but as a first step, we derive an important relationship between the functions $A$ and $T$.

\begin{Prop}
\label{050608}
If the Bellman candidate $B$ is $C^1$\!-smooth  and the corresponding foliation is a~left herringbone described above\textup, then
\eq{220502}{
2A'=(1-T')\frac{A-f_+}{\eps-T}+(1+T')\frac{A-f_-}{\eps+T}\,,
}
where $f_+=f_+(u+T(u)-\eps)$ and $f_-=f_-(u-T(u)-\eps)$.

In the case of a right herringbone\textup, we have\textup:
\eq{220503}{
2A'=(1-T')\frac{f_--A}{\eps+T}+(1+T')\frac{f_+-A}{\eps-T}\,,
}
where $f_+=f_+(u-T(u)+\eps)$ and $f_-=f_-(u+T(u)+\eps)$.
\end{Prop}

\begin{proof}
We consider only the case of the left herringbone. The symmetric case can be considered absolutely similarly; however,~\eqref{220503} can be formally obtained from~\eqref{220502} by the symmetry. To do this, we need to change the direction of the $x_1$-axis, i.\,e., change the sign of all first derivatives.

Consider the plane $P$ containing the following three points:
$$
\begin{aligned}
z&=(u,T,A),
\\
z_-&=(u-T-\eps,-\eps,f_-),
\\
z_+&=(u+T-\eps,\eps,f_+).
\end{aligned}
$$ 
The segments $[z,z_\pm]$ belong to the graph of the function $B$ because $B$ is linear in the corresponding directions. Since $B$ is diagonally concave and $C^1$\!-smooth, the plane $P$ is tangent to the graph of $B$ at the point $z$. Therefore, the vector $(1,T',A')$ that is tangent to the graph of $B$ along the spine of the herringbone is parallel to the plane $P$. Hence,
$$
\det
\begin{pmatrix}
\hskip 10pt 1\hskip 10pt  & T' & A'
\\
\eps+T & \eps+T & A-f_-\rule{0pt}{20pt}
\\
\eps-T & T-\eps &A-f_+\rule{0pt}{20pt}
\end{pmatrix}
=0\,,
$$
which coincides with~\eqref{220502}.
\end{proof}

Let us show that the equation~\eqref{220502} is not only a necessary but also a sufficient condition for the function~\eqref{220501} to be $C^1$\!-smooth. Before proving this, to simplify further calculations we introduce three auxiliary functions of the variable $u$:
\eq{020601}{
R_-\df\frac{A-f_-}{\eps+T},\qquad R_+\df\frac{A-f_+}{\eps-T},\quad\text{and}\quad R\df R_-+R_+.
}
We rewrite~\eqref{220502} as follows:
\eq{020602}{
2A'=(1-T')R_++(1+T')R_-,
}
and~\eqref{220501} takes the following form:
\eq{020603}{
B(x)=
\begin{cases}
(\eps-x_2)R_++f_+, 
&\text{if }x_2\ge T(u);
\\
(\eps+x_2)R_-+f_-,
&\text{if }x_2\le T(u).
\end{cases}
}
The meaning of the coefficients $R_\pm$ is evident from this formula: these are (up to the sign) the slopes of $B$ on the corresponding extremal lines.
\begin{Prop}
Under condition~\eqref{220502}\textup, the function $B$ defined by~\eqref{220501} is $C^1$\!-smooth.
\end{Prop}

\begin{proof}
We start with some auxiliary differentiation formulas that will be useful later. First, the relation
\eq{020604}{
\begin{cases}
x_1+x_2=u+T,
&\text{if }x_2\ge T(u);
\\
x_1-x_2=u-T, 
&\text{if }x_2\le T(u),
\end{cases}
}
leads to the identities
\eq{020605}{
u_{x_1}=
\begin{cases}
\frac1{1+T'},
&\text{if }x_2\ge T(u);
\\
\frac1{1-T'},
&\text{if }x_2\le T(u),
\end{cases}
\quad\text{and}\quad
u_{x_2}=
\begin{cases}
\phantom{-}\frac1{1+T'},
&\text{if }x_2\ge T(u);
\\
-\frac1{1-T'},
&\text{if }x_2\le T(u).
\end{cases}
}

Before differentiating $B$, let us calculate the derivatives of $R_+$ and $R_-$ with respect to $u$, using~\eqref{220502} in the equivalent form~\eqref{020602}.
\eq{070501}{
\begin{aligned}
R_+'&=\frac{A'-(1+T')f_+'}{\eps-T}+\frac{A-f_+}{(\eps-T)^2}T'
\\
&=\frac1{\eps-T}\Big(\half(1-T')R_++\half(1+T')R_--(1+T')f_+'+R_+T'\Big)
\\
&=\frac{1+T'}{\eps-T}\big(\half R-f_+'\big).
\end{aligned}
}
Here we took into account that
$$
\frac d{du}f_+ = \frac d{du}f_+\big(u+T(u)-\eps\big)=(1+T')f_+'.
$$

We introduce two more notations:
\eq{020606}{
N_+\df\frac{\half R-f_+'}{\eps-T}\quad\text{and}\quad N_-\df\frac{\half R-f_-'}{\eps+T}\,.
}
Then~\eqref{070501} can be rewritten as
\eq{020607}{
R_+'=(1+T')N_+.
}
Similarly, we  obtain 
\eq{020608}{
R_-'=(1-T')N_-.
}

We differentiate~\eqref{020603} with respect to $x_1$:
\eq{260501}{
B_{x_1}(x)=
\begin{cases}
(\eps-x_2)N_++f_+', 
&\text{if }x_2\ge T(u);
\\
(\eps+x_2)N_-+f_-',
&\text{if }x_2\le T(u).
\end{cases}
}
We see that $B_{x_1}$ is continuous from both sides of the spine, and the limit from any side at the point $(u,T)$ on the spine is equal to $\half R$, i.\,e., $B_{x_1}$ is continuous. Since the function $A$ is $C^1$\!-smooth, this is enough for the continuity of the gradient. Alternatively, it is also easy to directly verify the continuity of $B_{x_2}$.

Using~\eqref{020605}, we differentiate~\eqref{020603} with respect to~$x_2$:
\eq{020609}{
B_{x_2}(x)=
\begin{cases}
-R_++B_{x_1}(x), 
&\text{if }x_2\ge T(u);
\\
\phantom{-}R_--B_{x_1}(x),
&\text{if }x_2\le T(u).
\end{cases}
}
Therefore, $B_{x_2}$ is continuous from both sides of the spine, and the limit from any side at the point $(u,T)$ on the spine is equal to $\half(R_--R_+)$.
\end{proof}

The function defined by~\eqref{220501} is a Bellman candidate under the condition that it is concave in the directions $x_1\pm x_2=\const$. For $C^2$ functions, such concavity is equivalent to the inequalities $B_{x_1x_1}+B_{x_2x_2}\pm B_{x_1x_2}\le0$. We check these inequalities separately on each part of the herringbone, but thanks to $C^1$-smoothness, this implies the desired concavity on the entire herringbone.

\begin{Prop}
Let $B$ be defined by~\eqref{220501} using a function $A,$ satisfying the condition~\eqref{220502}. Then $B$ is concave in the directions $x_1\pm x_2=\const$ if and only if
$$
\max\{f_-'',f_+''\}\le\frac{f_+'-f_-'}{2T}\,,
$$
and $A$ is defined by $T$ and the boundary values $f_\pm$ as follows\textup:
\eq{040614}{
A=\frac{\eps^2-T^2}{2\eps T}\Big[(\eps+T)f_+'-(\eps-T)f_-'\Big]+\frac{(\eps+T)f_++(\eps-T)f_-}{2\eps}\,.
}
\end{Prop}

\begin{proof}
First, differentiate~\eqref{020609} with respect to $x_1$:
\eq{040601}{
B_{x_1x_2}=
\begin{cases}
-N_++B_{x_1x_1}, 
&\text{if }x_2\ge T(u),
\\
\phantom{-}N_--B_{x_1x_1},
&\text{if }x_2\le T(u).
\end{cases}
}
Then differentiate the same expression with respect to $x_2$:
\eq{040602}{
B_{x_2x_2}=
\begin{cases}
-N_++B_{x_1x_2}=-2N_++B_{x_1x_1}, 
&\text{if }x_2\ge T(u),
\\
-N_--B_{x_1x_2}=-2N_-+B_{x_1x_1},
&\text{if }x_2\le T(u).
\end{cases}
}
Therefore,
\eq{040603}{
B_{x_1x_1}+B_{x_2x_2}+2B_{x_1x_2}=
\begin{cases}
4\big(B_{x_1x_1}-N_+\big), 
&\text{if }x_2\ge T(u),
\\
\qquad\quad 0,
&\text{if }x_2\le T(u),
\end{cases}
}
and
\eq{040604}{
B_{x_1x_1}+B_{x_2x_2}-2B_{x_1x_2}=
\begin{cases}
\qquad\quad 0, 
&\text{if }x_2\ge T(u),
\\
4\big(B_{x_1x_1}-N_-\big),
&\text{if }x_2\le T(u).
\end{cases}
}

Thus, to prove the diagonal concavity of $B$ we need to calculate $B_{x_1x_1}$ and compare the obtained expression with $N_\pm$. First, consider the case where $x_2>T(u)$. Differentiating~\eqref{260501} with respect to $x_1$, we obtain
$$
\begin{aligned}
B_{x_1x_1}-N_+&=(\eps-x_2)N_+'u_{x_1}+f_+''-N_+,
\\
N_+'&=\frac{\half R_+'+\half R_-'-(1+T')f_+''}{\eps-T}+\frac{\half R-f_+'}{(\eps-T)^2}T'\rule{0pt}{25pt}
\\
&=\frac1{\eps-T}\big(\half(1+T')N_++\half(1-T')N_--(1+T')f_+''+N_+T'\big),
\end{aligned}
$$
from which
\eq{040605}{
B_{x_1x_1}-N_+=
\frac12\cdot\frac{\eps-x_2}{\eps-T}\Big(\frac{1+3T'}{1+T'}N_++\frac{1-T'}{1+T'}N_--2f_+''\Big)+f_+''-N_+.
}
This expression is linear on each extremal line (i.\,e., for fixed $u$), so it is non-positive for all values of $x_2$, $x_2 \in [T,\eps]$, if and only if it is non-positive at $x_2=\eps$ and $x_2=T$. Thus, we obtain the following necessary conditions:
\eq{040606}{
f_+''-N_+\le0;
}
\eq{040607}{
\frac{1-T'}{1+T'}(N_--N_+)\le0.
}

Now consider the lower half of the herringbone, where $x_2<T$. Similarly, we obtain
\eq{040608}{
B_{x_1x_1}-N_-=
\frac12\cdot\frac{\eps+x_2}{\eps+T}\Big(\frac{1+T'}{1-T'}N_++\frac{1-3T'}{1-T'}N_--2f_-''\Big)+f_-''-N_-,
}
from which the necessary conditions follow:
\eq{040609}{
f_-''-N_-\le0;
}
\eq{040610}{
\frac{1+T'}{1-T'}(N_+-N_-)\le0.
}
Note that conditions~\eqref{040606}, \eqref{040607}, \eqref{040609}, and~\eqref{040610} are not only necessary but also sufficient for~$B$ to be diagonally concave.

Comparing~\eqref{040607} and~\eqref{040610}, we obtain the necessary condition $N_+=N_-$ or
\eq{040611}{
(\eps+T)f_+'-(\eps-T)f_-'=TR=T\Big(\frac{A-f_-}{\eps+T}+\frac{A-f_+}{\eps-T}\Big),
}
which is equivalent to~\eqref{040614}. This equality defines~$A$ on the spine if a non-zero $T$ is given. 
When this condition is satisfied, we obtain
\eq{040617}{
N_+=N_-=\frac{1}{2T}\big((\eps+T)N_--(\eps-T)N_+\big) =\frac{f_+'-f_-'}{2T},
}
and the remaining conditions~\eqref{040606} and~\eqref{040609} transform into
\eq{040612}{
f_+''\le\frac{f_+'-f_-'}{2T},
}
\eq{040613}{
f_-''\le\frac{f_+'-f_-'}{2T}.
}
\end{proof}

\begin{Rem}
\label{120201}
Formula~\eqref{040611} implies that the left horizontal herringbone can intersect the $x_1$-axis only at a point $(u_0,0)$ for which $f'_+(u_0-\eps)=f'_-(u_0-\eps)$. We cannot calculate $A(u_0)$ directly from~\eqref{040614}. However, if such a point is isolated, then we can pass to the limit $T\to0$, $u\to u_0$ in~\eqref{040614} and find $A(u_0)$ in terms of boundary values and the slope of the spine at this point:
\eq{080701}{ 
A(u_0)=\frac12\big((f_++f_-)+\eps(f'_++f'_-)+\eps^2(f''_++f''_-)\big)
+\frac{\eps^2}{2T'(u_0)}(f''_+-f''_-).
}
Here all the values of $f_\pm$ and their derivatives are evaluated at~$u_0-\eps$. If $T'(u_0)=0$, then this formula does not work. This can only happen if $f''_+=f''_-$, i.\,e., $u_0-\eps$ is a multiple root of $f'_+-f'_-$. This case will be considered in one of the subsequent papers.

Symmetrically, the right horizontal herringbone can intersect the $x_1$-axis only at a point $(u_0,0)$ where $f'_+(u_0+\eps)=f'_-(u_0+\eps)$ {\bf(}see~\eqref{060604}{\bf)}. 

If the spine of a horizontal herringbone not only intersects the $x_1$-axis but coincides with it on some segment, then condition~\eqref{040611} shows that this is possible only if the functions $f_+$ and $f_-$ differ by a constant on the corresponding interval. For small $\eps$, the Bellman function inside the interval will resemble the one built for a symmetric strip, i.\,e., when $f_+=f_-$ everywhere (see~\cite{SymStr}). However, we are not ready to consider such a case and therefore postpone its investigation.
\end{Rem}

So now we have two equations~\eqref{040614} and~\eqref{220502} for two unknown functions $A$ and $T$. We consider~\eqref{040614} as the definition of~$A$ in terms of $T$. The function $T$ is defined by the differential equation~\eqref{040615}, which we derive by substituting~\eqref{040614} into~\eqref{220502}.
\begin{Prop}
If $T$ defines the spine of the left horizontal herringbone in the foliation of a~diagonally concave function with boundary values $f_\pm$\textup, then it satisfies the following differential equation
\eq{040615}{
T'=\frac{(\eps-T)D\ut{L}_--(\eps+T)D\ut{L}_+}{(\eps-T)D\ut{L}_-+(\eps+T)D\ut{L}_+},
}
where
\eq{040616}{
D\ut{L}_+=\frac{f_+'-f_-'}{2T}-f_+''\ge0\quad\text{and}\quad 
D\ut{L}_-=\frac{f_+'-f_-'}{2T}-f_-''\ge0\,.
}
\end{Prop}
\begin{proof}
We differentiate~\eqref{040611} and get
\eq{090501}{
(f_+'+f_-')T'+(\eps+T)(1+T')f_+''-(\eps-T)(1-T')f_-''=T'R+TR'.
}
Using~\eqref{020607}, \eqref{020608}, and~\eqref{040617}, we obtain
$$
TR'=T(1+T')N_++T(1-T')N_-=f_+'-f_-',
$$
and~\eqref{040611} can be written as
$$
R=f_+'+f_-'+\frac\eps{T}(f_+'-f_-').
$$
Substituting these formulas into~\eqref{090501} and collecting terms with $T'$, we obtain
$$
T'\big[(\eps+T)f_+''+(\eps-T)f_-''-\frac\eps{T}(f_+'-f_-')\big]=(\eps-T)f_-''-(\eps+T)f_+''+f_+'-f_-',
$$
or 
$$
T'\big[-(\eps+T)D\ut{L}_+-(\eps-T)D\ut{L}_-\big]=-(\eps-T)D\ut{L}_-+(\eps+T)D\ut{L}_+,
$$
which coincides with~\eqref{040615}.
\end{proof}

We conclude this section by listing the changes in the formulas for a left horizontal herringbone that need to be made to obtain the corresponding formulas for a right herringbone. As we mentioned at the beginning of the proof of Proposition~\ref{050608}, we need to change the direction of the $x_1$-axis, thus we formally reverse the sign of all first derivatives. To distinguish right and left herringbones, we will use the indices $\RR$ and $\LL$. For example, $T\ti{R}$ and $T\ti{L}$ are functions defining the spines of the right and left herringbones, $A\ti{R}$ and $A\ti{L}$ are the values of $B$ on the corresponding spines, and so on.

Until the end of this section, we consider the right herringbone, so all objects can be marked with the index $\RR$, but usually we omit these indices.

A subtle change has to be made: the argument of $f_+$ and its derivatives is $u-T+\eps=x_1-x_2+\eps$, and the argument of the function $f_-$ is $u+T+\eps=x_1+x_2+\eps$. The function $u=u(x)$ satisfies the relation
\eq{050605}{
\begin{cases}
x_1-x_2=u-T,
&\text{if }x_2\ge T(u),
\\
x_1+x_2=u+T, 
&\text{if }x_2\le T(u),
\end{cases}
}
instead of~\eqref{020604}. Therefore, instead of~\eqref{020605} we obtain 
\eq{050606}{
u_{x_1}=
\begin{cases}
\frac1{1-T'},
&\text{if }x_2\ge T(u),
\\
\frac1{1+T'},
&\text{if }x_2\le T(u),
\end{cases}
\quad\text{and}\quad
u_{x_2}=
\begin{cases}
-\frac1{1-T'},
&\text{if }x_2\ge T(u),
\\
\phantom{-}\frac1{1+T'},
&\text{if }x_2\le T(u).
\end{cases}
} 
Definition~\eqref{220501} of~$B$ remains the same, as well as the definition~\eqref{020601} of $R_\pm$ and $R$. We change the signs of all derivatives in~\eqref{020602}, and it takes the form
\eq{050607}{
2A'=-(1-T')R_--(1+T')R_+.
}
This formula is given in Proposition~\ref{050608}, see~\eqref{220503}.

We make a change in the definition of $N_\pm$ by changing the signs of derivatives. Instead of~\eqref{020606} we write
\eq{050609}{
N_+\ut{R}\df-\frac{\half R+f_+'}{\eps-T}\quad\text{and}\quad N_-\ut{R}\df-\frac{\half R+f_-'}{\eps+T}.
}
Then, instead of~\eqref{020607} and~\eqref{020608} we get
\eq{050610}{
R_+'=(1-T')N_+\qquad\text{and}\qquad R_-'=(1+T')N_-.
}
Formula~\eqref{260501} for $B_{x_1}$ remains the same, and in~\eqref{020609} we need to change the sign of $B_{x_1}$:
\eq{050611}{
B_{x_2}=
\begin{cases}
-R_+-B_{x_1}, 
&\text{if }x_2\ge T(u),
\\
\phantom{-}R_-+B_{x_1},
&\text{if }x_2\le T(u).
\end{cases}
}

Instead of conditions~\eqref{040603} and~\eqref{040604} we obtain
\eq{060601}{
B_{x_1x_1}+B_{x_2x_2}+2B_{x_1x_2}=
\begin{cases}
\qquad\quad 0, 
&\text{if }x_2\ge T(u),
\\
4\big(B_{x_1x_1}+N_-\big),
&\text{if }x_2\le T(u),
\end{cases}
}
and
\eq{060602}{
B_{x_1x_1}+B_{x_2x_2}-2B_{x_1x_2}=
\begin{cases}
4\big(B_{x_1x_1}+N_+\big),
&\text{if }x_2\ge T(u),
\\
\qquad\quad 0,
&\text{if }x_2\le T(u).
\end{cases}
}

The condition of diagonal concavity leads to the same relation $N_+=N_-$, giving us the definition of $A$, where now there will be a different sign in front of $R$. Instead of~\eqref{040611}, we obtain
\eq{060603}{
(\eps-T)f_-'-(\eps+T)f_+'=TR=T\Big(\frac{A-f_-}{\eps+T}+\frac{A-f_+}{\eps-T}\Big),
}
from which
\eq{060604}{
A=\frac{\eps^2-T^2}{2\eps T}\Big[(\eps-T)f_-'-(\eps+T)f_+'\Big]+\frac{(\eps+T)f_++(\eps-T)f_-}{2\eps}
}
instead of~\eqref{040614}. The expression~\eqref{040617} for $N_\pm$ remains the same
\eq{060605}{
N_+=N_-=\frac{f_+'-f_-'}{2T},
}
but we must change the signs in the conditions of diagonal concavity~\eqref{040612} and~\eqref{040613}:
\eq{060606}{
f_+''\le\frac{f_-'-f_+'}{2T},
}
\eq{060607}{
f_-''\le\frac{f_-'-f_+'}{2T}.
}
In this situation, it is natural to introduce functions
\eq{141201}{
D_+\ut{R}=\frac{f_-'-f_+'}{2T}-f_+''\qquad\text{and}\qquad D_-\ut{R}=\frac{f_-'-f_+'}{2T}-f_-''\,.
}
With this definition, the condition of diagonal concavity remains unchanged: $D_\pm\ut{R}\ge0$. Finally, the analogue of~\eqref{040615} is the following equation:
\eq{060612}{
T'\ti{R}=\frac{(\eps+T\ti{R})D_+\ut{R}-(\eps-T\ti{R})D_-\ut{R}}{(\eps+T\ti{R})D_+\ut{R}+(\eps-T\ti{R})D_-\ut{R}}\,.
}

So far, we have used the notations $D_\pm\ut{R}$ and $D_\pm\ut{L}$ for functions of the variable $u$. Now we introduce functions $D_\pm$ defined in the entire plane except the $x_1$-axis.
\begin{Def}
\eq{060608}{
D_+(x)\df\frac{f_+'(x_1+x_2)-f_-'(x_1-x_2)}{2x_2}-f_+''(x_1+x_2);
}
\eq{060609}{
D_-(x)\df\frac{f_+'(x_1+x_2)-f_-'(x_1-x_2)}{2x_2}-f_-''(x_1-x_2).
}
\end{Def}

The functions $D_\pm\ut{R}$ and $D_\pm\ut{L}$ defined above for the case of a right or a left herringbone with the functions $T\ti{R}$ and $T\ti{L}$ respectively, can be expressed in terms of $D_\pm$ as follows:
$$
D_\pm\ut{L}(u)=D_\pm(u-\eps,T\ti{L}(u)), \qquad D_\pm\ut{R}(u)=D_\pm(u+\eps,-T\ti{R}(u)).
$$

Now we can state that the conditions $D_\pm\ge0$ are necessary and sufficient for the diagonal concavity not only on the herringbones but also on the regions $\Om{R}$ and $\Om{L}$. The condition~\eqref{190501} of concavity on $\Om{R}$ can be rewritten as $D_\pm(u,\eps)\ge0$, and the condition~\eqref{200502} of concavity on $\Om{L}$ as $D_\pm(u,-\eps)\ge 0$.

\section{Investigation of the vector field of equation~\eqref{040615}}
\label{Sec5}

In this section, we explore the behavior of integral curves of the vector field of the differential equation~\eqref{040615}
describing left herringbones. We shift the first coordinate by $\eps$: instead of $u-\eps$, we write $x_1$, while $T$ remains the second coordinate $x_2$. Thus, for the left herringbone, the argument of $f_+$ will be $x_1+x_2$ instead of $u+T-\eps$, and the argument of $f_-$ will be $x_1-x_2$ instead of $u-T-\eps$. This is more convenient for consideration of the evolution of integral curves when the parameter $\eps$ increases. If $f'_+(u_0)=f'_-(u_0)$ for some $u_0$, then the node $(u_0,0)$ does not move as $\eps$ grows. The corresponding point of the intersection of the left herringbone with the central line of the strip is $(u_0+\eps,0)$, i.\,e., the left herringbone moves to the right.

In the following, we consider the parametrization $x=x(t)$ of the investigated integral curves of the system
\eq{180901}{
\begin{cases}
2\dot x_1=(\eps-x_2)\big[f_+'(x_1+x_2)-f_-'(x_1-x_2)-2x_2f_-''(x_1-x_2)\big]
\\
\qquad+(\eps+x_2)\big[f_+'(x_1+x_2)-f_-'(x_1-x_2)-2x_2f_+''(x_1+x_2)\big],
\\
2\dot x_2=(\eps-x_2)\big[f_+'(x_1+x_2)-f_-'(x_1-x_2)-2x_2f_-''(x_1-x_2)\big]
\\
\qquad-(\eps+x_2)\big[f_+'(x_1+x_2)-f_-'(x_1-x_2)-2x_2f_+''(x_1+x_2)\big].
\end{cases}
}

To write shorter formulas, we omit the arguments of $f_\pm$ and their derivatives, assuming that the functions with the sign $\pm$ always have the argument $x_1\pm x_2$. Thus, instead of~\eqref{180901}, we write
\eq{180902}{
\begin{cases}
\dot x_1=(\eps-x_2)\big[\half(f_+'-f_-')-x_2f_-''\big]+(\eps+x_2)\big[\half(f_+'-f_-')-x_2f_+''\big],
\\
\dot x_2=(\eps-x_2)\big[\half(f_+'-f_-')-x_2f_-''\big]-(\eps+x_2)\big[\half(f_+'-f_-')-x_2f_+''\big]
\end{cases}
}
or
\eq{260901}{
\begin{cases}
\dot x_1=\eps(f_+'-f_-')-x_2(\eps-x_2)f_-''-x_2(\eps+x_2)f_+'',
\\
\dot x_2=x_2(f_-'-f_+')-x_2(\eps-x_2)f_-''+x_2(\eps+x_2)f_+''.
\end{cases}
}

We start with the description of stationary points of the system, i.\,e., points where the right-hand side of the equation is zero. First, consider stationary points lying on the central line $x_2=0$. They have the form $(u_0,0)$, where $u_0$ is determined by the equation
\eq{011002}{
f_+'(u_0)-f_-'(u_0)=0.
}
The intervals of the central line between these stationary points are integral curves of the vector field. Any left herringbone can intersect the central line of the strip only at points $(u_0+\eps,0)$, where $u_0$ is a solution to~\eqref{011002}. Symmetrically, the points of intersection of right herringbones with the central line can only have the form $(u_0-\eps,0)$.

To describe the behavior of an integral curve in a neighborhood of a stationary point, we calculate the Jacobian matrix of system~\eqref{180902}:
\eq{011004}{
J(x)\!=\!
\begin{pmatrix}
\scriptstyle\eps(f_+''-f_-'')-x_2(\eps-x_2)f_-'''-x_2(\eps+x_2)f_+'''&
\scriptstyle 2x_2(f_-''-f_+'')+x_2(\eps-x_2)f_-'''-x_2(\eps+x_2)f_+'''
\\
\scriptstyle x_2(f_-''-f_+'')-x_2(\eps-x_2)f_-'''+x_2(\eps+x_2)f_+'''&
\scriptstyle(f_-'-f_+')-(\eps-x_2)(f_-''-x_2f_-''')+(\eps+x_2)(f_+''+x_2f_+''')
\end{pmatrix}.
}

Let $u_0$ be a root of~\eqref{011002}. Then at the stationary point $x=(u_0,0)$ we have
$$
J(u_0,0)=\eps\big((f_+''(u_0)-f_-''(u_0)\big)\begin{pmatrix}
1&0
\\
0&1
\end{pmatrix}.
$$
Furthermore, we assume that equation~\eqref{011002} has only a finite number of simple roots. Therefore, $f_+''(u_0)-f_-''(u_0)\ne0$, and the Jacobian matrix is an identity operator up to a scalar multiplier. This means that the stationary point $(u_0,0)$ is always a dicritical node, and the local behavior of integral curves is shown in Figure~\ref{081202}\footnote{In local pictures we draw the behavior of integral curves of the linearized system, which differs from the original system by a diffeomorphism with a unit Jacobian matrix}.

\begin{figure}[h]
    \centering
    \includegraphics[scale = 0.6]{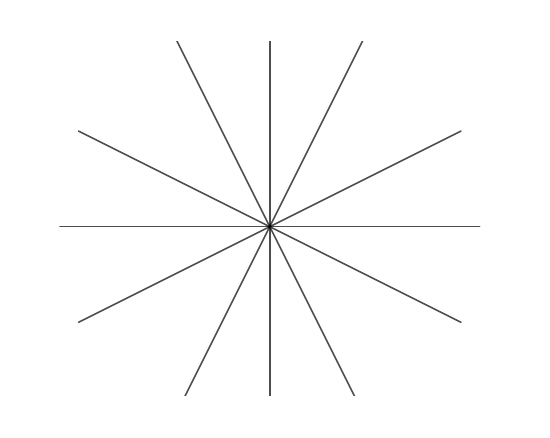}
    \caption{Integral curves near the node $(u_0,0)$.}
    \label{081202}
\end{figure}

As mentioned earlier, we assume the following conditions: $f_+'(u_0)=f_-'(u_0)$ and $f_+''(u_0)\ne f_-''(u_0)$. Consider two curves in a neighborhood of $(u_0,0)$, that are defined by equations~\eqref{300901} and~\eqref{300902}:
\eq{300901}{
f_+'(x_1+x_2)-f_-'(x_1-x_2)-2x_2f_+''(x_1+x_2)=0,
}
(this means $D_+=0$ for $x_2\ne 0$) and
\eq{300902}{
f_+'(x_1+x_2)-f_-'(x_1-x_2)-2x_2f_-''(x_1-x_2)=0,
}
(this means $D_-=0$ for $x_2\ne 0$). If $f_\pm$ are smooth, then each of the equations~\eqref{300901} and~\eqref{300902} has a unique smooth solution $x_1=x_1(x_2)$ in a neighborhood of $(u_0,0)$. This follows from the implicit function theorem and the fact that the derivatives of the left-hand sides of~\eqref{300901} and~\eqref{300902} at $(u_0,0)$ with respect to $x_1$ are equal to $f_+''(u_0)-f_-''(u_0) \ne 0$. These curves intersect at the node $(u_0,0)$. Later, we will see that for small $\eps$, the intersection of the curve~\eqref{300901} with the boundary $x_2=\eps$ and the intersection of the curve~\eqref{300902} with the boundary $x_2=-\eps$ give us two more stationary points of the vector field: $(u_+,\eps)$ and $(u_-,-\eps)$.

\begin{Prop}
\label{230701}
The curve $D_+(x)=0$ in a neighborhood of $(u_0,0)$ is described by the equation
\eq{230702}{
x_1=u_0+x_2+\vk_+x_2^2+O(x_2^3)\,.
}
The curve $D_-(x)=0$ in a neighborhood of $(u_0,0)$ is described by the equation
\eq{230703}{
x_1=u_0-x_2+\vk_-x_2^2+O(x_2^3)\,.
}
Here
\eq{230704}{
\vk_\pm=\pm\frac{2f_\pm'''(u_0)}{f_+''(u_0)-f_-''(u_0)}\,.
}
The sign of the function $D_\pm$ to the right of the curve $D_\pm=0$ coincides with the sign of the expression $x_2(f_+''(u_0)-f_-''(u_0)),$ and to the left\textup, the signs are opposite.
\end{Prop}

\begin{proof}
With a fixed $x_2\ne 0$, consider $D_+$ as a function of $x_1$ on the interval $[u_0,u_0+2x_2]$. If $x_2$ is small enough, this function has opposite signs at the endpoints of the interval:
$$
\begin{aligned}
D_+(u_0,x_2)&=\frac{f_+'(u_0)+x_2 f_+''(u_0)-f_-'(u_0)+x_2 f_-''(u_0)}{2x_2}-f_+''(u_0)+O(x_2)
\\
&=\half\big(f_-''(u_0)-f_+''(u_0)\big)+O(x_2)\,,
\\
D_+(u_0+2x_2,x_2)&=
\frac{f_+'(u_0)+3x_2 f_+''(u_0)-f_-'(u_0)-x_2 f_-''(u_0)}{2x_2}-f_+''(u_0)+O(x_2)\rule{0pt}{20pt}
\\
&=\half\big(f_+''(u_0)-f_-''(u_0)\big)+O(x_2)\,.
\end{aligned}
$$
Therefore, inside the interval, there exists a root $x_1=x_1(x_2)$. The uniqueness of the root is guaranteed by the implicit function theorem.

The sign of the function $D_+$ to the right of the curve $D_+=0$ coincides with the sign of the derivative
$$
\frac{\partial D_+(x)}{\partial x_1}=\frac{f_+''(x_1\!+x_2)\!-\!f_-''(x_1\!-\!x_2)}{2x_2}-f_+'''(x_1\!+x_2) 
= \frac{f_+''(u_0)\!-\!f_-''(u_0)}{2x_2}+O(1),
$$
and to the left, the sign is opposite.

To prove the representation~\eqref{230702}, we expand the function $x_1(x_2)$ up to a second-order terms in~$x_2$: 
$x_1=u_0+c_1x_2+c_2x_2^2+O(x_2^3)$.
Then,
$$
\begin{aligned}
0&=D_+(x_1,x_2)=\frac1{2x_2}\Big(f_+'(u_0)+f_+''(u_0)\big((c_1+1)x_2+c_2x_2^2\big)
\\
&\quad+\half f_+'''(u_0)\big((c_1+1)x_2\big)^2+O(x_2^3)
\\
&\quad-f_-'(u_0)-f_-''(u_0)\big((c_1-1)x_2+c_2x_2^2\big)-\half f_-'''(u_0)\big((c_1-1)x_2\big)^2+O(x_2^3)\Big)
\\
&\quad-f_+''(u_0)-(c_1+1)x_2 f_+'''(u_0)+O(x_2^2)
\\
&=\frac{c_1-1}2\big(f_+''(u_0)-f_-''(u_0)\big)+\Big(\frac{c_2}2\big(f_+''(u_0)-f_-''(u_0)\big)
\\
&\quad+\frac{c_1^2-2c_1-3}4f_+'''(u_0)-\frac{(c_1-1)^2}4f_-'''(u_0)\Big)x_2+O(x_2^2).
\end{aligned}
$$
The constant term gives $c_1=1$, and then the coefficient of $x_2$ gives $c_2=\varkappa_+$.

The proof of the second part of the proposition is completely similar: instead of $D_+$ we consider~$D_-$.
\end{proof}

Let us calculate the slope of the curve $D_+=0$:
$$
d(2x_2D_+)=(f_+''-f_-''-2x_2f_+''')dx_1+(f_+''+f_-''-2f_+''-2x_2f_+''')dx_2=0.
$$
Therefore,
$$
\frac{dx_2}{dx_1}=\frac{f_+''-f_-''-2x_2f_+'''}{f_+''-f_-''+2x_2f_+'''}=
1-\frac{4x_2f_+'''}{f_+''-f_-''+2x_2f_+'''}=1-2\vk_+x_2+O(x_2^2).
$$
Similarly, the slope of the curve $D_-=0$ is given by:
$$
\frac{dx_2}{dx_1}=-1+2\vk_-x_2+O(x_2^2).
$$
Thus, we make the following conclusion.
\begin{Rem}\label{rem150901}
For small $\eps$ the slope of the curve  $D_+=0$ is positive and if $\vk_+>0$, then it is less than $1$ in the upper half of the strip and greater than $1$ in the lower half. For small $\eps$ the slope of the curve $D_-=0$ is negative.
\end{Rem}

\begin{Cor}
\label{230705}
If $\eps$ is sufficiently small\textup, then there exists a unique 
root $u_+$  $(u_+\ut{L})$ of the equation $D_+(u,\eps)=0,$ i.\,e.\textup,
\eq{230706}{
f_+'(u+\eps)-f_-'(u-\eps)-2\eps f_+''(u+\eps)=0,
}
such that $u_+\in (u_0,u_0+2\eps)$. Moreover\textup,
\eq{230707}{
u_+=u_0+\eps+\vk_+\eps^2+O(\eps^3)\,.
}
Similarly\textup, there exists a unique root $u_-$ $(u_-\ut{L})$  of the equation $D_-(u,-\eps)=0,$ i.\,e.\textup,
\eq{230708}{
f_-'(u+\eps)-f_+'(u-\eps)-2\eps f_-''(u+\eps)=0,
}
such that $u_-\in (u_0,u_0+2\eps)$. Moreover\textup,
\eq{230709}{
u_-=u_0+\eps+\vk_-\eps^2+O(\eps^3)\,.
}
\end{Cor}

The following simple statement plays an important role when investigating the behavior of the integral curves of the vector field~\eqref{180902}.

\begin{Prop}
\label{210901}
The slope of the integral curve of the field~\eqref{180902} at a point~$x$ is strictly increasing in~$\eps$ if $x_2D_+(x)D_-(x)>0,$ 
and strictly decreases if $x_2D_+(x)D_-(x)<0$.
\end{Prop}

\begin{proof}
From~\eqref{180902} we obtain the following expression for the slope of the integral curve:
\eq{210902}{
\frac{\dot x_2}{\dot x_1}=\frac{(\eps-x_2)D_--(\eps+x_2)D_+}{{(\eps-x_2)D_-+(\eps+x_2)D_+}}\,.
}
Differentiating this expression with respect to $\eps$ gives a formula that directly implies the desired result:
\eq{210903}{
\frac{\partial}{\partial\eps}\Big(\frac{\dot x_2}{\dot x_1}\Big)=\frac{4x_2D_-D_+}{\big[(\eps-x_2)D_-+(\eps+x_2)D_+\big]^2}\,.
}
\end{proof}

Now we investigate the vector field around the stationary point $(u_+,\eps)$. Since $u_+$ satisfies~\eqref{230706}, formula~\eqref{011004} takes the form
\eq{210101}{
J(u_+,\eps)=\eps\big(f_+''(u_++\eps)-f_-''(u_+-\eps)-2\eps f_+'''(u_++\eps)\big)
\begin{pmatrix}
\phantom{i}1&-2-3\eps\kappa_+
\\
-1&\eps\kappa_+
\end{pmatrix},
}
where
\eq{180503}{
\kappa_+=\frac{2f_+'''(u_++\eps)}{f_+''(u_++\eps)-f_-''(u_+-\eps)-2\eps f_+'''(u_++\eps)}\,,
}
and $\kappa_+\to\varkappa_+$ as $\eps\to0$.
For small values of the parameter $\eps$, we have 
$$
J(u_+,\eps)=\eps\big(f_+''(u_0)-f_-''(u_0)\big)
\begin{pmatrix}
\phantom{i}1&\!\!-2\;
\\
-1&0
\end{pmatrix} \ +\  O(\eps^2).
$$
The eigenvalues of the matrix on the right-hand side are $-1$ and $2$, i.\,e., this is a saddle point. The eigenvectors are $(\begin{smallmatrix}1\\1\end{smallmatrix})$ and $(\begin{smallmatrix}2\\\!\!-1\end{smallmatrix})$, meaning
two integral curves passing through this stationary point have slopes $1$ and $-\half$ (up to $O(\eps)$). The behavior of integral curves around the stationary point $(u_+,\eps)$ is shown in Fig.~\ref{101201}.

\begin{figure}[h]
    \centering
    \includegraphics[scale = 0.5]{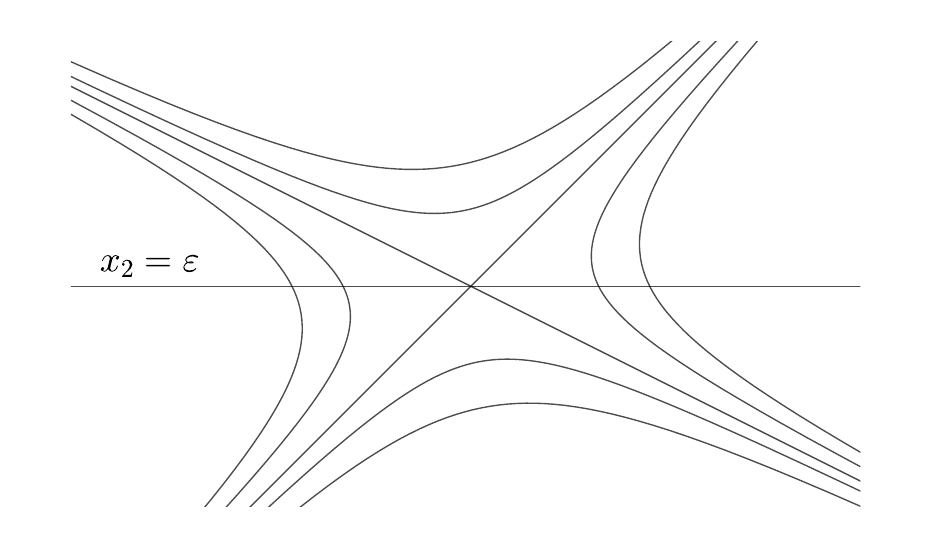}
    \caption{Integral curves near the saddle point $(u_+\ut{L},\eps)$.}
    \label{101201}
\end{figure}

We need a more detailed approximation of the eigenvector of~\eqref{210101} corresponding to the eigenvalue
$\lambda$, which approaches $-1$ as $\eps\to0$. It is easy to compute that 
$\lambda = -1-\frac13\varkappa_+\eps+O(\eps^2)$, and the corresponding eigenvector is equal to
\eq{220101}{
\left(\begin{matrix}1\\1-\frac43\varkappa_+\eps\end{matrix}\right)
}
up to $O(\eps^2)$. 
Thus the slope of the corresponding integral curve at $(u_+,\eps)$ is 
$1-\frac43\vk_+\eps + O(\eps^2)$.
It will be important to compare the slope of this curve with the slope of the curve $D_+(x)=0$ at $(u_+,\eps)$. We use~\eqref{230702} to calculate the slope of the curve $D_+(x)=0$:
\eq{240701}{
\frac{dx_2}{dx_1}=1-2\vk_+x_2+O(x_2^2)\,.
}
Thus, the relative position of these two curves near $(u_+,\eps)$ is determined by the sign of~$\vk_+$.

Now consider the stationary point $(u_-,-\eps)$. Since $u_-$ satisfies~\eqref{230708},
formula~\eqref{011004} takes the form
$$
J(u_-,-\eps)=\eps\big(f_+''(u_--\eps)-f_-''(u_-+\eps)+2\eps f_-'''(u_-+\eps)\big)
\begin{pmatrix}
1&2+3\eps\kappa_-
\\
1&\eps\kappa_-
\end{pmatrix},
$$
where
$$
\kappa_-=\frac{2f_-'''(u_-+\eps)}{f_-''(u_-+\eps)-f_+''(u_--\eps)-2\eps f_-'''(u_-+\eps)}\,.
$$
Therefore, for small $\eps$ we have
$$
J(u_-,-\eps)\approx\eps\big(f_+''(u_0)-f_-''(u_0)\big)
\begin{pmatrix}
1&2
\\
1&0
\end{pmatrix}.
$$
The matrix has the same eigenvalues $-1$ and $2$, so this is also a saddle point. Now the eigenvectors are
$(\begin{smallmatrix}1\\-1\end{smallmatrix})$ and $(\begin{smallmatrix}2\\1\end{smallmatrix})$, 
meaning two integral curves passing through this stationary point have slopes $-1$ and $\frac12$ (up to $O(\eps)$). 
The behavior of integral curves near the stationary point $(u_-,-\eps)$ is shown in Fig.~\ref{141207}.

\begin{figure}[h]
    \centering
    \includegraphics[scale = 0.05]{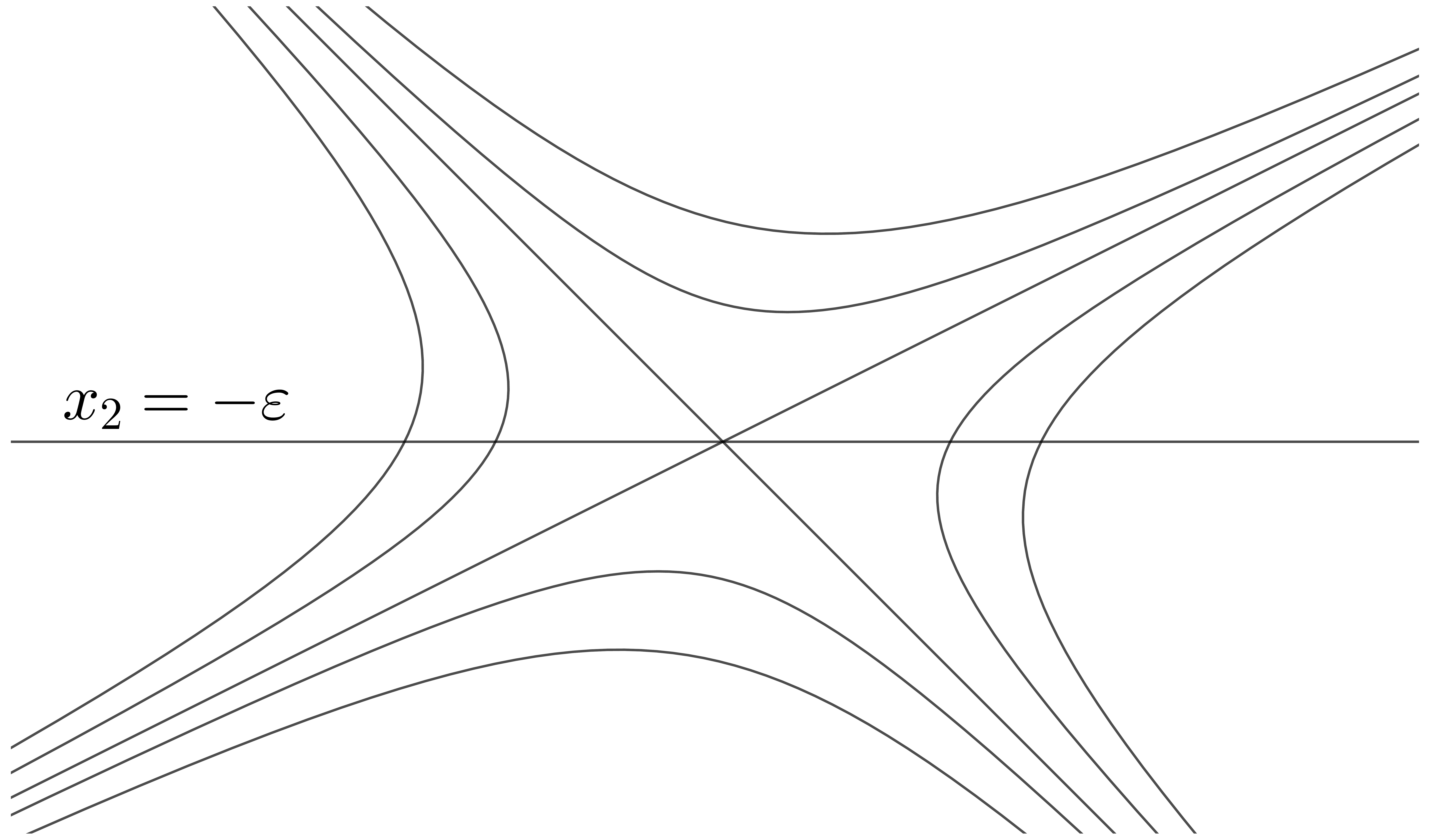}
    \caption{Integral curves near the saddle point $(u_-\ut{L},-\eps)$.}
    \label{141207}
\end{figure}

Combining all this information, we now understand the behavior of integral curves of the vector field near 
these three stationary points for small values of $\eps$. The overall picture is shown in Fig.~\ref{141208}.

\begin{figure}[h]
    \centering
    \includegraphics[scale = 0.4]{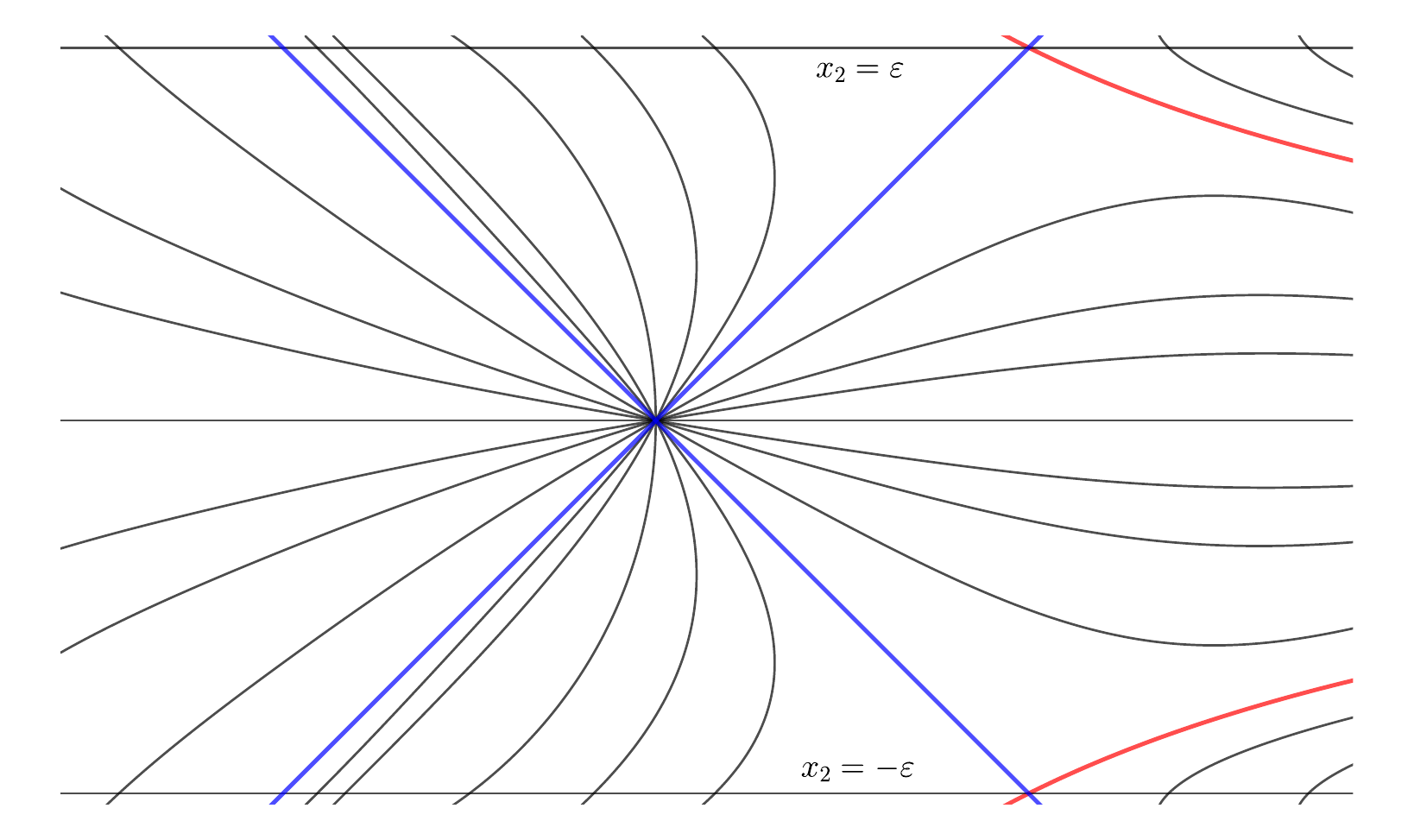}
    \caption{Integral curves near the node $(u_0,0)$ for small values of $\eps$.}
    \label{141208}
\end{figure}

At the end of this section, we list the changes that need to be made if we consider 
the vector field of the equation~\eqref{060612}, the integral curve of which may be the spine of a right herringbone. In this case, 
we will have a symmetric picture. For sufficiently small $\eps$, we have the following three stationary points: $(u_0,0)$, $(u_+\ut{R},\eps)$, and $(u_-\ut{R},-\eps)$, where $u_+\ut{R}$ is the unique solution of the equation
\eq{201202}{
f_+'(u-\eps)-f_-'(u+\eps)+2\eps f_+''(u-\eps)=0
}
in the interval $(u_0-2\eps,u_0)$, and $u_-\ut{R}$ is the unique solution of the equation
\eq{201203}{
f_-'(u-\eps)-f_+'(u+\eps)+2\eps f_-''(u-\eps)=0
}
in the same interval. Furthermore,
\eq{181201}{
u_\pm\ut{R}=u_0-\eps+\varkappa_\pm\eps^2+o(\eps^2)\,,
}
where $\varkappa_\pm$ are defined in~\eqref{230704}.

\section{Fissures}
\label{Fissures}

We are ready to construct a foliation that we call a fissure. First, let us consider left fissures. Recall that these are horizontal herringbones that extend from left to right from one boundary to the other. 
Any fissure must intersect the middle line of the strip, meaning there must be a value $u$ such that $T(u)=0$. 
From the previous section, we know that the point $(u,0)$ must be a stationary point, where $f_+'(u-\eps)=f_-'(u-\eps)$. 
This can also be seen directly from~\eqref{040611}. \big(For the right herringbone, this condition is $f_+'(u+\eps)=f_-'(u+\eps)$\big). 
Thus, these are precisely the points around which we could not construct either the right or the left simple foliation for small $\eps$, see Section~\ref{Sec2}.

In this section, for sufficiently small $\eps$, we build a fissure near a point $u_0$ that solves the equation~\eqref{210501}. 
The SW-fissure is shown in Fig.~\ref{181203}. The term ``SW-fissure'' means that the fissure comes from the southwest. In other words, it is a horizontal herringbone extending from left to right from the bottom boundary to the top.  We denote the region foliated by this fissure as $\Om{SW}(v,u)$. Here, $u$ and $v$ are the first coordinates of the endpoints of the segment of the intersection of the foliated region with the middle line.
\begin{figure}[h]
    \centering
    \includegraphics[scale = 0.4]{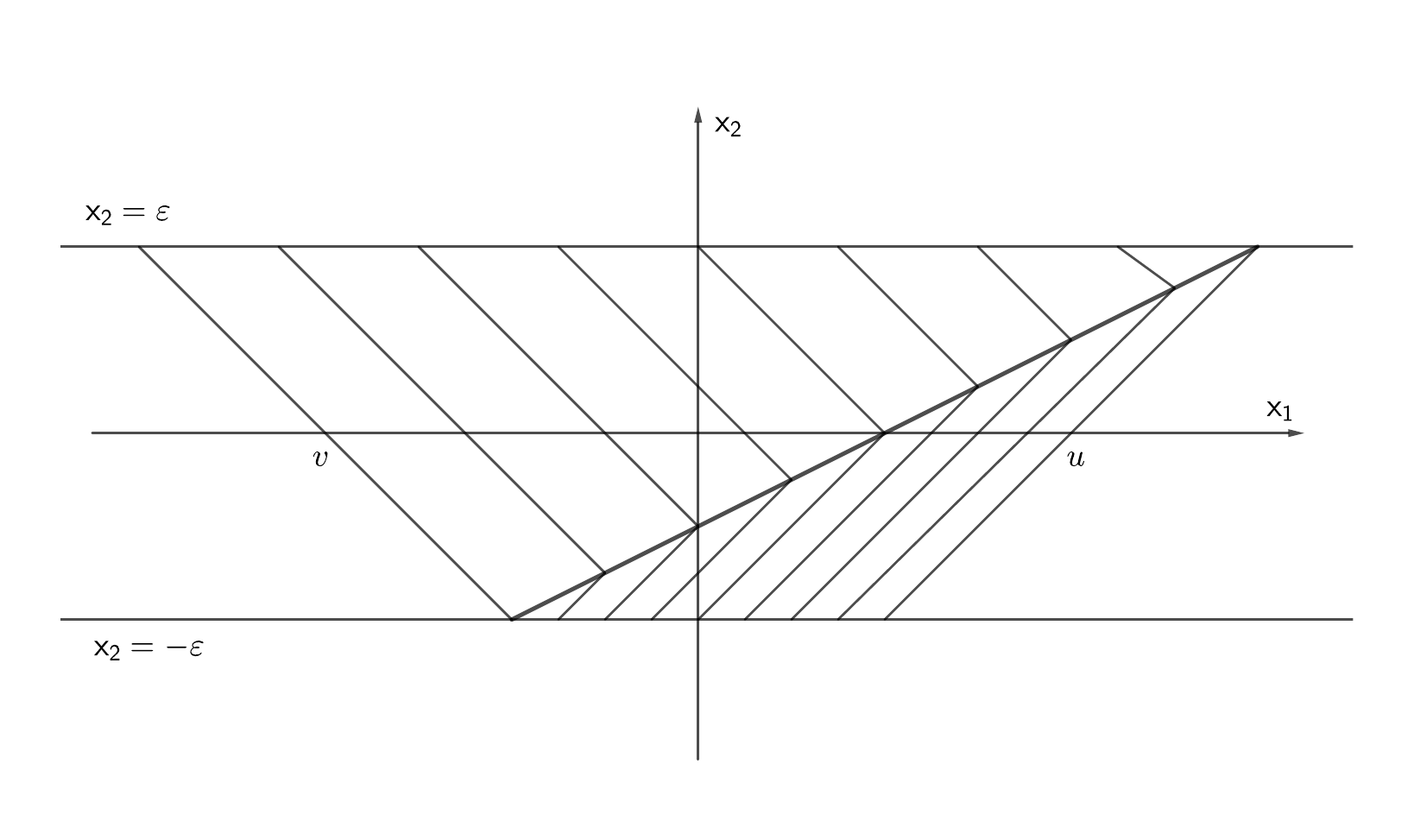}
    \caption{SW-fissure on its region $\Om{SW}(v,u)$.}
    \label{181203}
\end{figure}

This foliation may appear not for all possible boundary values. Therefore, we impose the following additional conditions:

\begin{itemize}
\item
$u_0$ is a simple root of the equation~\eqref{210501}, meaning 
$$
f_+'(u_0)=f_-'(u_0)\qquad\text{and}\qquad f_+''(u_0)\ne f_-''(u_0);
$$
\item
$$
f_+'''(u_0)\ne0, \qquad\text{if}\qquad f_+''(u_0)>f_-''(u_0),
$$
\vskip-5pt
and
\vskip-15pt
$$
f_-'''(u_0)\ne0, \qquad\text{if}\qquad f_+''(u_0)<f_-''(u_0).
$$
\end{itemize}
\smallskip
Foliations that appear when these additional conditions are not satisfied are more complicated and will be considered later.

Our goal is to prove the following statement.
\begin{Prop}
\label{200101}
Let $f_+'(u_0)=f_-'(u_0)$\textup, $f_+''(u_0)>f_-''(u_0)$\textup, $f_+'''(u_0)>0$\textup, and let 
$u_+\!=u_+\ut{L}$ be the root of the equation~\eqref{230706} described in Corollary~\textup{\ref{230705}}. 
Then for sufficiently small $\eps,$ there exists a~\textup{SW}-fissure with a spine function $T=T\ti{L}$ defined on some interval 
$[v_- +\eps, u_+ +\eps]$ such that 
\begin{itemize}
\item
$T(u_0+\eps)=0$\textup, \ $T(u_+\!+\eps)=\eps$\textup;
\item 
$v_-\!=v_-\ut{L}$\textup, \;$v_-\!\!<u_0-\eps$\textup, 
solves the equation 
$T(v_-\!+\eps)=-\eps$\textup;
\item
Formula~\eqref{220501} with the function $A=A\ti{L}$ given by~\eqref{040614} 
defines a Bellman candidate on~$\Om{SW}(v_-,u_+)$ foliated by this fissure\textup;
\item
There exists a $\delta=\delta(\eps)$\textup, such that the standard Bellman candidates on 
$\Om{R}(u_+,u_+\!+\delta)$ and $\Om{L}(v_-\!-\delta,v_-)$ \textup(see~\eqref{180501} and~\eqref{200501}\textup) 
together with the described candidate on $\Om{SW}(v_-,u_+)$ 
form a $C^1$\!-smooth Bellman candidate on the union of these three regions.
\end{itemize}
\end{Prop}

The signs of the second and third derivatives determine the direction of the corresponding fissure. Thus, reflection with respect to the line $x_1=u_0$ gives us a SE-fissure.

\begin{Prop}
\label{200102}
Let $f_+'(u_0)=f_-'(u_0)$\textup, $f_+''(u_0)>f_-''(u_0)$\textup, $f_+'''(u_0)<0$\textup, and let  $u_+=u_+\ut{R}$ be the root of the equation~\eqref{201202} described in Corollary~\textup{\ref{230705}}. Then for sufficiently small $\eps,$ there exists a~\textup{SE}-fissure with a spine function $T=T\ti{R}$ defined on some interval
$[v_- -\eps, u_+ -\eps]$ such that 
\begin{itemize}
\item
$T(u_0-\eps)=0$\textup, \ $T(u_+\!-\eps)=\eps$\textup;
\item
$v_-\!=v_-\ut{R}$\textup, \;$v_-\!\!>u_0+\eps$\textup, solves the equation  
$T(v_-\!-\eps)=-\eps$\textup;
\item
Formula~\eqref{220501} with the function $A=A\ti{R}$ given by~\eqref{060604} 
defines a Bellman candidate on~$\Om{SE}(u_+,v_-)$\textup, foliated by this fissure\textup;
\item
There exists a $\delta=\delta(\eps)$\textup, such that the standard Bellman candidates on  
$\Om{R}(u_+\!-\delta,u_+)$ and $\Om{L}(v_-\!-\delta,v_-)$ \textup(see~\eqref{180501} and~\eqref{200501}\textup) 
together with the described candidate on $\Om{SE}(u_+,v_-)$ 
form a $C^1$\!-smooth Bellman candidate on the union of these three regions.
\end{itemize}
\end{Prop}

The reflection with respect to the line $x_2=0$ swaps $f_+$ and $f_-$. This leads to the following statements.

\begin{Prop}
\label{200103}
Let $f_+'(u_0)=f_-'(u_0)$\textup, $f_+''(u_0)<f_-''(u_0)$\textup, $f_-'''(u_0)>0$\textup, and let  $u_-\!=u_-\ut{L}$ be the root of the equation~\eqref{230708} described in Corollary~\textup{\ref{230705}}. 
Then for sufficiently small $\eps,$ there exists a~\textup{NW}-fissure with a spine function $T=T\ti{L}$ defined on some interval $[v_+ +\eps, u_- +\eps]$ such that
\begin{itemize}
\item
$T(u_0+\eps)=0$\textup, \ $T(u_-\!+\eps)=-\eps$\textup;
\item
$v_+=v_+\ut{L}$\textup, \;$v_+\!<u_0-\eps$\textup, solves the equation 
$T(v_+\!+\eps)=\eps$\textup;
\item
Formula~\eqref{220501} with the function $A=A\ti{L}$ given by~\eqref{040614} defines a Bellman candidate on~$\Om{NW}(v_+,u_-)$ foliated by this fissure\textup;
\item
There exists a $\delta=\delta(\eps)$\textup, such that the standard Bellman candidates on  
$\Om{R}(u_-,u_-\!+\delta)$ and $\Om{L}(v_+\!-\delta,v_+)$ \textup(see~\eqref{180501} and~\eqref{200501}\textup)
together with the described candidate on $\Om{NW}(v_+,u_-)$ form a $C^1$\!-smooth Bellman candidate on the union of these three regions.
\end{itemize}
\end{Prop}

\begin{Prop}
\label{200104}
Let $f_+'(u_0)=f_-'(u_0)$\textup, $f_+''(u_0)<f_-''(u_0)$\textup, $f_+'''(u_0)<0$\textup, and let
 $u_-=u\ut{R}_-$ be the root of the equation~\eqref{201203} described in Corollary~\textup{\ref{230705}}. 
 Then for sufficiently small $\eps,$ there exists a~\textup{NE}-fissure with a spine function $T=T\ti{R}$ defined on some interval $[v_+ -\eps, u_- -\eps]$ such that
\begin{itemize}
\item
$T(u_0-\eps)=0$\textup, \ $T(u_-\!-\eps)=-\eps$\textup;
\item
$v_+=v_+\ut{R}$\textup, \;$v_+\!>u_0+\eps$ solves the equation  
$T(v_+\!-\eps)=\eps$\textup;
\item
Formula~\eqref{220501} with the function $A=A\ti{R}$ given by~\eqref{060604} defines a Bellman candidate on~$\Om{NE}(u_-,v_+)$ foliated by this fissure\textup;
\item
There exists a $\delta=\delta(\eps)$\textup, such that the standard Bellman candidates on 
$\Om{R}(u_-\!-\delta,u_-)$ and $\Om{L}(v_+,v_-\!+\delta)$ \textup(see~\eqref{180501} and~\eqref{200501}\textup)
together with the described candidate on $\Om{NE}(u_+,v_-)$ form a $C^1$\!-smooth Bellman candidate on the union of these three regions.
\end{itemize}
\end{Prop}

We will consider only the case of an SW-fissure in a neighborhood of $u_0$, where $f_+'(u_0)=f_-'(u_0)$,
$f_+''(u_0)>f_-'(u_0)$, and $f_+'''(u_0)>0$. To do this, we need to solve the equation~\eqref{040615}.
We seek a solution $T$ on some interval $[v_-\!+\eps,u_+\!+\eps]$ with boundary conditions 
$T(v_-\!+\eps)=-\eps$, \hbox{$T(u_+\!+\eps)=\eps$} and satisfying the convexity conditions~\eqref{040616}.

\begin{proof}[Proof of Proposition~\textup{\ref{200101}}]

Investigating the vector field~\eqref{180901} in the preceding section, we claimed that there exists
exactly one integral curve of this field connecting the node $(u_0,0)$ to the saddle point $(u_+,\eps)$
on the upper boundary. Now we will prove this statement and describe more detailed properties
of this curve.

As we recall, the field~\eqref{180901} is horizontally shifted by $\eps$ compared to the field generated by the equation~\eqref{040615}. The mentioned integral curve
will give us the upper half of the herringbone between the points $(u_0+\eps,0)$ and $(u_+\!+\eps,\eps)$. To obtain the lower half of the herringbone, we must take the unique integral curve starting at the same node, but directed towards the lower boundary, such that together with the upper half they form a smooth $C^1$ curve in the natural parametrization (see one of the bold curves in Figure~\ref{141208}). 
We only need to verify that the convexity conditions~\eqref{040616} are satisfied. They ensure that 
\begin{itemize}
\item the slope of the herringbone spine does not exceed $1$, i.\,e., we can construct the desired foliation (see Figure~\ref{181203});
\item the corresponding function is diagonally concave.
\end{itemize}

Let $\ell_\eps$ denote the integral curve of the field~\eqref{180902} originating at  $(u_+,\eps)$ with a positive slope. We know that such a curve exists and is unique. Moreover, its slope at $(u_+,\eps)$ is described by~\eqref{220101}:
\eq{210904}{
1-\tfrac43\vk_+\eps+O(\eps^2)\,.
}

The imposed conditions on the derivatives of the functions $f_\pm$ ensure $\vk_+>0$. From~\eqref{240701} we see that the slope of the curve $D_+=0$ at $(u_+,\eps)$ is $1-2\vk_+\eps+O(\eps^2)$, i.\,e., is less than the slope of the curve $\ell_\eps$. As noted in Proposition~\ref{230701}, the region where $D_+>0$ lies to the right of the curve $D_+=0$ in the upper half of the strip and to the left of this curve in the lower half of the strip (see Figure~\ref{220102}).
\begin{figure}[h]
    \centering
    \includegraphics[scale = 0.5]{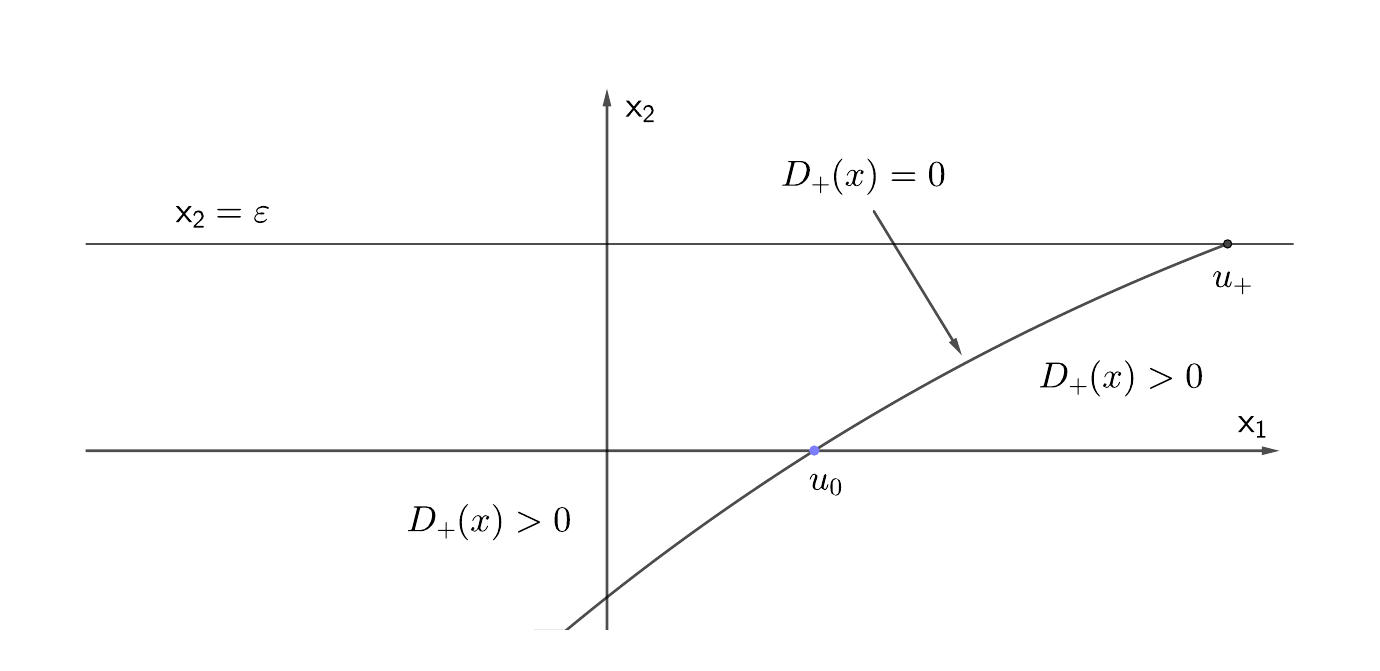}
    \caption{Regions where $D_+>0$.}
    \label{220102}
\end{figure}

Consider the open subdomain $\omega$ of the upper half of the strip bounded by the curve $D_+=0$, the middle line $x_2=0$, and the vertical line $x_1=u_+$. We have  $D_+> 0$ on $\omega$. As we explained, the curve $\ell_\eps$ originating at the saddle point $(u_+,\eps)$, in a neighborhood of this point lies below the curve $D_+=0$, so it enters $\omega$. We need to prove that it crosses the boundary of $\omega$ again at the node $(u_0,0)$. 

Note that due to the assumption $f_+''(u_0)>f_-''(u_0)$, the inequality $f_+'' > f_-''$ is satisfied in some neighborhood of $(u_0,0)$, thus for sufficiently small $\eps$, it holds in $\omega$ as well. The inequality $f_+''>f_-''$ for $x_2>0$ implies $D_->D_+$, hence $D_->0$ is also satisfied in~$\omega$. Then, $\dot x_1>0$ (see~\eqref{180902}). This means that the integral curve of the vector field~\eqref{180902} in~$\omega$ is the graph of some function of the first coordinate. 

The slope of any integral curve of the field~\eqref{180902} is equal to $1$ at the points of intersection with the curve $D_+=0$ (except, possibly, stationary points). The slope of the curve $D_+=0$ in the upper half of the strip is strictly less than $1$, see~\eqref{240701}. Thus the integral curves in the upper half of the strip intersect the curve $D_+=0$ from top to bottom when moving from right to left. Therefore, the integral curve $\ell_\eps$ cannot exit  $\omega$ through the boundary $D_+=0$ in the upper half of the strip. 
Furthermore, two integral curves of the vector field can intersect only at stationary points, 
i.\,e., the middle line and $\ell_\eps$ can only intersect at the node $(u_0,0)$. Thus, we have proved that~$\ell_\eps$ connects two stationary points: the saddle point $(u_+,\eps)$ and the node $(u_0,0)$ {\bf(}see Figure~\ref{250101}{\bf)}.
\begin{figure}[h]
    \centering
    \includegraphics[scale = 0.25]{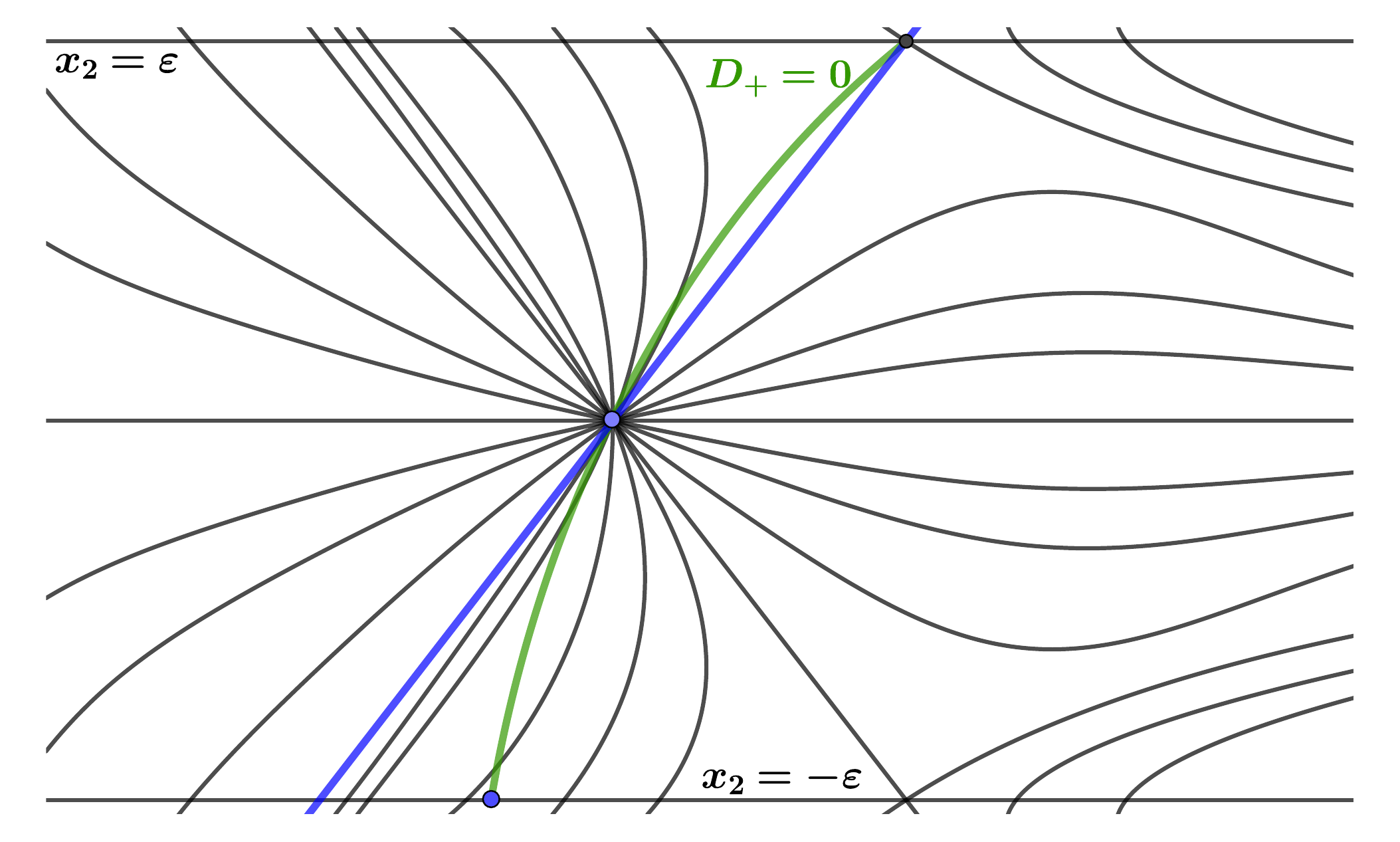}
    \caption{Integral curves and the curve $D_+=0$.}
    \label{250101}
\end{figure}

Fix $\eps>0$ and temporarily consider $\delta\leq\eps$. The slope of the curve $\delta\mapsto(u_+(\delta),\delta)$ is equal to
$$
\frac1{1+2\vk_+\delta+O(\delta^2)}=1-2\vk_+\delta+O(\delta^2)
$$
{\bf(}see~\eqref{230707}{\bf)}, while the slope of the curve $\ell_\eps$ at $(u_+(\eps),\eps)$ is 
$1-\frac43\vk_+\eps+O(\eps^2)$ {\bf(}see~\eqref{210904}{\bf)}, meaning the latter curve is steeper in a neighborhood of this point. Therefore, if $\delta$ is sufficiently close to $\eps$, then $(u_+(\delta),\delta)$ lies above $\ell_\eps$. Consequently, the curve $\ell_\delta$ also lies above $\ell_\eps$ in some neighborhood of this point. Now, suppose that $\ell_\delta$ intersects $\ell_\eps$ at some point. Since it lies above $\ell_\eps$ until the first point of intersection, its slope at the point of intersection must be not less than the slope of~$\ell_\eps$, which is impossible within $\omega$ according to Proposition~\ref{210901}. Due to the continuity of the curve~$\ell_\delta$ with respect to $\delta$, we conclude that for any $\delta<\eps$, $\ell_\delta$ lies to the left (and above) of $\ell_\eps$. Thus, all curves $\ell_\eps$ intersect only at the node $(u_0,0)$.

We have shown that $D_\pm>0$ on any integral curve $\ell_\eps$ with $x_2>0$, thus the slope of $\ell_\eps$ is strictly less than $1$ in the upper half of the strip, and now we need to verify this at the node $(u_0,0)$. In fact, we need to check that $\ell_\eps$ is not the integral curve with the unit slope at the node. For this purpose, we will calculate this integral curve up to second order (in $x_2$) and compare it with the curve $D_+=0$, i.\,e., with~\eqref{230702}.

Let the integral curve we are interested in be defined by the equation $x_1=u_0+x_2+\alpha x_2^2+O(x_2^3)$. Then, from~\eqref{260901} we obtain
$$
\begin{aligned}
\frac{\dot x_1}{x_2}&=\eps\frac{f_+'(u_0+2x_2+\alpha x_2^2)-f_-'(u_0+\alpha x_2^2)}{x_2}
\\
&\quad-(\eps-x_2)f_-''(u_0)-(\eps+x_2)f_+''(u_0+2x_2)+O(x_2^2)
\\
&=\eps\big[(2+\alpha x_2)f_+''(u_0)+2x_2 f_+'''(u_0)-\alpha x_2 f_-''(u_0)\big]\rule{0pt}{18pt}
\\
&\quad-(\eps-x_2)f_-''(u_0)-(\eps+x_2)\big(f_+''(u_0)+2x_2 f_+'''(u_0)\big)+O(x_2^2)
\\
&=\big(f_+''(u_0)-f_-''(u_0)\big)\big(\eps-(1-\alpha\eps)x_2\big)+O(x_2^2)\rule{0pt}{18pt}
\end{aligned}
$$
and
$$
\begin{aligned}
\frac{\dot x_2}{x_2}&=f_-'(u_0)-f_+'(u_0+2x_2)-(\eps-x_2)f_-''(u_0)+(\eps+x_2)f_+''(u_0+2x_2)+O(x_2^2)
\\
&=-2x_2 f_+''(u_0)-(\eps-x_2)f_-''(u_0)+(\eps+x_2)\big(f_+''(u_0)+2x_2 f_+'''(u_0)\big)+O(x_2^2)
\\
&=\big(f_+''(u_0)-f_-''(u_0)\big)\big(\eps-(1-\vk_+\eps)x_2\big)+O(x_2^2)\,.
\end{aligned}
$$
Recall that $\vk_+$ is given by~\eqref{230704}.

For a smooth integral curve $x_1=u_0+x_2+\alpha x_2^2+O(x_2^3)$ we have
$$
\dot x_1=\big[1+2\alpha x_2+O(x_2^2)\big]\dot x_2\,,
$$
thus
$$
(1+2\alpha x_2)\big(\eps-(1-\vk_+\eps)x_2\big)=\eps-(1-\alpha\eps)x_2+O(x_2^2)\,
$$
which implies $\alpha=-\vk_+$.

Comparing the obtained equation of the investigated integral curve $x_1=u_0+x_2-\vk_+ x_2^2+O(x_2^3)$ with the equation of the curve $D_+=0$ {\bf(}i.\,e., with~\eqref{230702}: $x_1=u_0+x_2+\vk_+ x_2^2+O(x_2^3)${\bf)}, we see that this integral curve goes to the left (above) of the curve $D_+=0$, i.\,e., in the domain where $D_+<0$, and therefore cannot coincide with $\ell_\eps$.

We conclude that the slope of $\ell_\eps$ at $(u_0,0)$ is strictly less than $1$. To continue it to the lower half of the strip, we choose an integral curve starting from $(u_0,0)$ to the left with the same slope. Thus, we obtain a $C^1$-smooth curve $\ell_\eps$ in the entire strip.

The slope of the curve $D_+=0$ in the lower half of the strip is strictly greater than $1$, see Remark~\ref{rem150901}. Reasoning similar to that given for the upper half of the strip shows that the curve~$\ell_\eps$ in the lower half of the strip has no intersections with the curve $D_+=0$ and lies in the domain where $D_+>0$.

We have proved that $D_+\ge 0$ at all points of the curve $\ell_\eps$. The second condition ($D_-\ge 0$) is automatically satisfied because
\eq{130201}{
D_- - D_+ = f_+'' - f_-'' = f_+''(u_0) - f_-''(u_0) + O(\eps) > 0
}
for $\eps$ sufficiently small.

Now, we prove that for small $\eps$ there exists a point of intersection of the curve $\ell_\eps$ with the lower boundary of the strip. We denote the first coordinate of this point by $v_-$. We will check that $v_--u_0=O(\eps)$.

Fix a $\delta>0$ and consider the curve $\ell_\delta$ in the domain $x_2<0$. Fix a $\eta>0$ at which this curve intersects the line $x_2=-\eta$ and consider $\eps\le\eta$. For the same reasons as in the upper half of the strip, the curves $\ell_\eps$ do not intersect. The curve with a larger $\eps$ is located above the curve with a smaller $\eps$. Therefore, for all $\eps\le\eta$, there exists a point $(v_-(\eps),-\eps)$ between $(u_0,-\eps)$ and the point of intersection of the curve $\ell_\delta$ with $x_2=-\eps$. Since the slope of $\ell_\delta$ at $(u_0,0)$ is strictly between $0$ and $1$, the first coordinate of a point $x$ on $\ell_\delta$ is $O(x_2)$, and consequently, $v_--u_0=O(\eps)$.

Thus, we have constructed a SW-fissure $T$ in the domain $\Omega_{\text{SW}}(v_-,u_+)$. The point $u_+$ is the root of the equation~\eqref{230706}. To construct the required domain $\Omega_{\text{R}}(u_+,u_++\delta)$, it is sufficient to satisfy the conditions~\eqref{190501}. Due to~\eqref{130201}, we only need to check the first condition:
\eq{130202}{
f_+'(u+\eps) - f_-'(u-\eps) - 2\eps f_+''(u+\eps) \ge 0\,
}
if $u \in (u_+,u_++\delta)$ for some positive $\delta$. Since the expression in~\eqref{130202} is $0$ when $u=u_+$ by definition of $u_+$ (see~\eqref{230706}), and its derivative
$$
f_+''(u+\eps) - f_-''(u-\eps) - 2\eps f_+'''(u+\eps) = f_+''(u_0) - f_-''(u_0) + O(\eps)
$$
is strictly positive for small $\eps$, we conclude that~\eqref{130202} holds for some $\delta=\delta(\eps)$ if $\eps$ is sufficiently small.

Unfortunately, we cannot say anything similar about $v_-$. But since the slope of the curve $\ell_\eps$ inside the strip lies strictly between $0$ and $1$, the inequality
\eq{240101}{
v_-<u_0-\eps
}
is always true. This is sufficient to construct the simple left foliation on some domain $\Omega_{\text{L}}(v_--\delta,v_-)$ to the left of the fissure. To do this, we need to check that the conditions~\eqref{200502} are satisfied in some left neighborhood of~$v_-$. Again, due to~\eqref{130201}, it suffices to check only the first condition, i.\,e.,
\eq{130203}{
f_-'(u+\eps) - f_+'(u-\eps) - 2\eps f_+''(u-\eps) \ge 0\,,
}
if $u \in (v_--\delta,v_-)$ for some $\delta$, $\delta=O(\eps)$. Expand the expression~\eqref{130203} at $u_0$:
\eq{140201}{
f_-'(u+\eps) - f_+'(u-\eps) - 2\eps f_+''(u-\eps) = -(u+\eps-u_0)[f_+''(u_0)-f_-''(u_0)]+O(\eps^2)\,.
}
Since for $u \in (v_--\delta,v_-)$ we have $u+\eps-u_0<v_-+\eps-u_0<0$,
the expression in~\eqref{140201} is strictly positive for sufficiently small $\eps$. 
\end{proof}

\section{Examples. Polynomials of third degree}

In this section, we partially consider the case when the boundary functions $f_\pm$ are third-degree polynomials. We do not consider all possible polynomials since we have not described all foliations that can arise for third-degree polynomials. However, the considered foliations allow us to describe all cases when $\eps$ is sufficiently small. Other possible foliations will be considered in the forthcoming papers.

Let
$$
f_\pm(t)=a_3^\pm t^3+a_2^\pm t^2+a_1^\pm t+a_0^\pm.
$$
We start with the case $a_3^+=a_3^-=a_3\neq0$. The case $a_3^+=a_3^-=0$ is considered in Section~\ref{230201}.
The symmetry with respect to the $x_2$-axis allows us to consider only the case $a_3>0$ (otherwise, replace~$t$ with~$-t$).
Using symmetry with respect to the $x_1$-axis, we can assume $a_2^+\geq a_2^-$ (otherwise, swap~$f_+$ and~$f_-$).

First, consider the simplest case $a_2^+=a_2^-$. If $a_1^+=a_1^-$, then $f_+-f_-$ is a constant, and, therefore, the foliation is the same as in the symmetric case; we have a horizontal herringbone, the spine of which is the $x_1$-axis (see~\cite{SymStr}). Since we assumed $a_3>0$, it will be the left herringbone. In the case  $a_3<0$, of course, it will be a symmetric right herringbone. Thus, we further assume $a_1^+\neq a_1^-$.

Under this assumption, we can construct a simple foliation for small $\eps$. Indeed, by direct calculation, condition~\eqref{190501} transforms into
\eq{120202}{
\begin{gathered}
(a_1^+-a_1^-)-12a_3\eps^2\geq0\,;
\\
(a_1^+-a_1^-)+12a_3\eps^2\geq0\,,
\end{gathered}
}
and condition~\eqref{200502} takes the form
\eq{120203}{
\begin{gathered}
(a_1^--a_1^+)+12a_3\eps^2\geq0\,;
\\
(a_1^--a_1^+)-12a_3\eps^2\geq0\,.
\end{gathered}
}

For small $\eps$, condition~\eqref{120202} is satisfied if $a_1^+>a_1^-$. More precisely,~\eqref{120202} is satisfied for $\eps\leq\eps_0$, where
\eq{120204}{
\eps_0\df\sqrt{\frac{|a_1^+-a_1^-|}{12a_3}}.
} 
Thus, for such $\eps$ we have the right simple foliation $\Om{R}(-\infty,+\infty)$. If $a_1^+<a_1^-$ and $\eps\leq\eps_0$, then condition~\eqref{120203} is satisfied and we have the left simple foliation $\Om{L}(-\infty,+\infty)$.

For $\eps>\eps_0$, we construct an infinite left herringbone, the spine of which is the line
$$
x_2=\frac{a_1^+-a_1^-}{12a_3\eps}\,.
$$
We consider the case $a_1^+>a_1^-$, and the opposite case can be obtained by swapping $f_+$ and $f_-$.

In fact, we can explicitly solve equation~\eqref{040615} since it takes the following simple form:
\eq{120205}{
T'=\frac{T(\eps T-\eps_0^2)}{\eps\eps_0^2-T^3}\,.
}
All solutions of this equation are represented by the family
\eq{120206}{
u=-\frac{T^2}{2\eps}-\frac{\eps_0^2}{\eps^2}T-\eps\log|T|+
\frac{\eps^4-\eps_0^4}{\eps^3}\log\big|T-\frac{\eps_0^2}{\eps}\big|+\const
}
and two special solutions
$$
T=0\qquad\text{and}\qquad T=\frac{\eps_0^2}{\eps}\,.
$$

We know that a herringbone with $T=0$ is impossible (see Remark~\ref{120201}), some of the herringbones with spines given by~\eqref{120206} generate diagonally concave functions, but the minimal diagonally concave function is generated by the spine $T=\eps_0^2\eps^{-1}$.

To check this, we compute the functions $D_\pm$ (see~\eqref{060608} and~\eqref{060609}). In this simple case, direct calculations give:
\eq{140201a}{
D_\pm=\frac{a_1^+-a_1^-}{2x_2}\mp6a_3x_2=\frac{6a_3}{x_2}(\eps_0^2\mp x_2^2)\,.
}
For the herringbone spines to be admissible for diagonally concave function, the integral lines must lie in the region where $D_\pm\ge0$, i.\,e., in the strip $0< x_2\le\eps_0$. Consider the extremal lines of the field generated by equation~\eqref{120205}. It is easy to do by plotting the graph of function~\eqref{120206} for any value of the constant (see Fig.~\ref{130204}).

\begin{figure}[h]
    \centering
   \includegraphics[scale = 0.55]{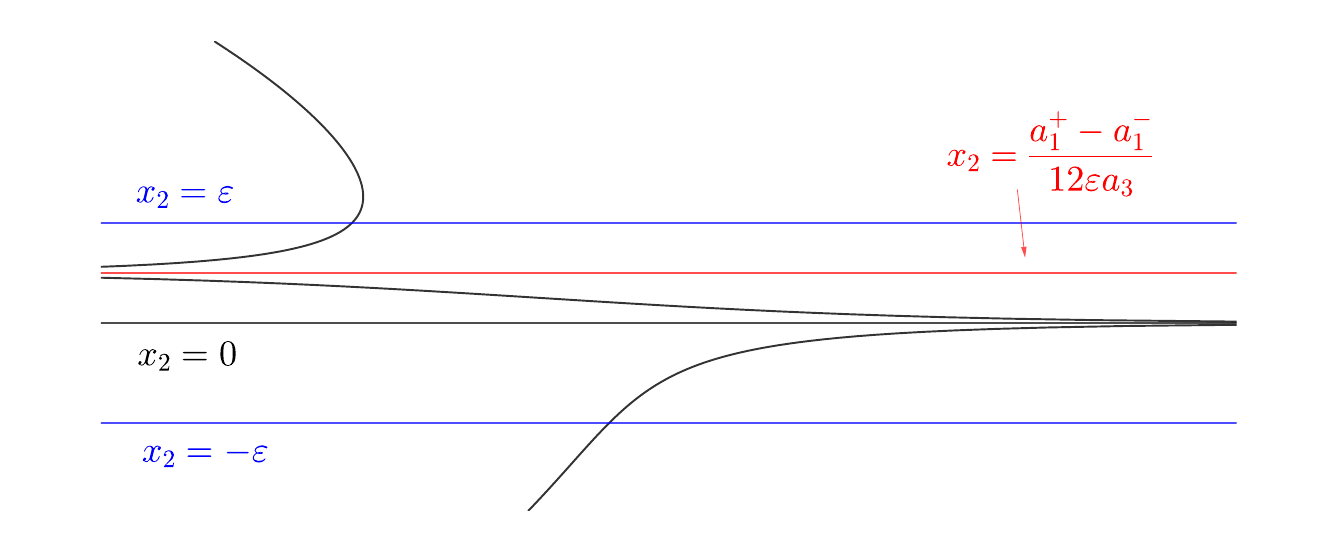}
    \caption{The graph of one of the curves~\eqref{120206}.}
    \label{130204}
\end{figure}

To obtain an integral line passing through a given point, we must shift this graph appropriately along the $x_1$-axis. We see that the sought extremal lines foliate the strip between the two horizontal asymptotes: $0 < x_2 < \eps_0^2\eps^{-1}$. If $T\in(0,\eps_0^2\eps^{-1})$, then from~\eqref{120205} we get $T'<0$ . Moreover, $T$ exponentially decays at infinity: \eqref{120206} shows that $T(u)\sim\exp(-u/\eps)$ as $u\to +\infty$. Then, using~\eqref{040614}, we conclude that $A(u)$ increases exponentially as $u\to\infty,$ whereas for the herringbone with $T(u)=\eps_0^2\eps^{-1}$, $A$ increases polynomial. This means that only the herringbone with $T(u)=\eps_0^2\eps^{-1}$ can give the minimal diagonally concave function. Indeed, the corresponding function is the minimal diagonally concave function according to Theorem~\ref{MishasTh}.

Concluding our consideration of the case $a_2^+=a_2^-$, note that if $a_1^+<a_1^-$, then we obtain the Bellman function by taking the left herringbone with $T(u)=-\eps_0^2\eps^{-1}$. Right herringbones appear when $a_3<0$.

Now let us turn to the case $a_2^+\ne a_2^-$. In this situation, the expression
$$
f_+'(t)-f_-'(t)=2(a_2^+-a_2^-)t+(a_1^+-a_1^-)
$$
has one root at
\eq{020301}{
u_0=-\frac{a_1^+-a_1^-}{2(a_2^+-a_2^-)}\,,
}
and we can construct a fissure. Under our assumptions ($a_3>0$, $a_2^+>a_2^-$), this will be a SW-fissure, as the assumptions of Proposition~\ref{200101} are satisfied.

In addition to Proposition~\ref{200101}, we want to check that such a foliation takes place for all~$\eps$, i.\,e., the foliation of the entire strip is as follows:
\eq{230202}{
\Omega_\eps=\Om{L}(-\infty,v_-)\bigcup\Om{SW}(v_-,u_+)\bigcup\Om{R}(u_+,+\infty)\,.
}
To check this, we need to verify three facts for all $\eps$:
\begin{itemize}
\item we can construct a SW-fissure in the region $\Om{SW}(v_-,u_+)$;
\item we can construct a simple right foliation in the region $\Om{R}(u_+,+\infty)$, i.\,e., condition~\eqref{190501} is satisfied for $u\ge u_+$;
\item we can construct a simple left foliation in the region $\Om{L}(-\infty,v_-)$, i.\,e., condition~\eqref{200502} is satisfied for $u\le v_-$.
\end{itemize}

For this purpose, we need to examine the vector field~\eqref{180902} more closely. It is convenient to consider the field not only in the strip $\Omega_\eps$ but also on the entire plane. The considered vector field has four stationary points. Three of them are already known: $(u_0,0)$, $(u_+,\eps)$, and $(u_-,-\eps)$. The fourth point is the second intersection of the two parabolas $D_+=0$ and $D_-=0$ given by the equations
\eq{070302}{
x_1=u_0+x_2+\vk_+x_2^2
}
and
\eq{070302bis}{
x_1=u_0-x_2-\vk_+ x_2^2
}
respectively (see Fig.~\ref{040301}). Direct calculation shows that it has coordinates $(u_0,-\vk_+^{-1})$.

\begin{figure}[h]
    \centering
    \includegraphics[scale = 0.25]{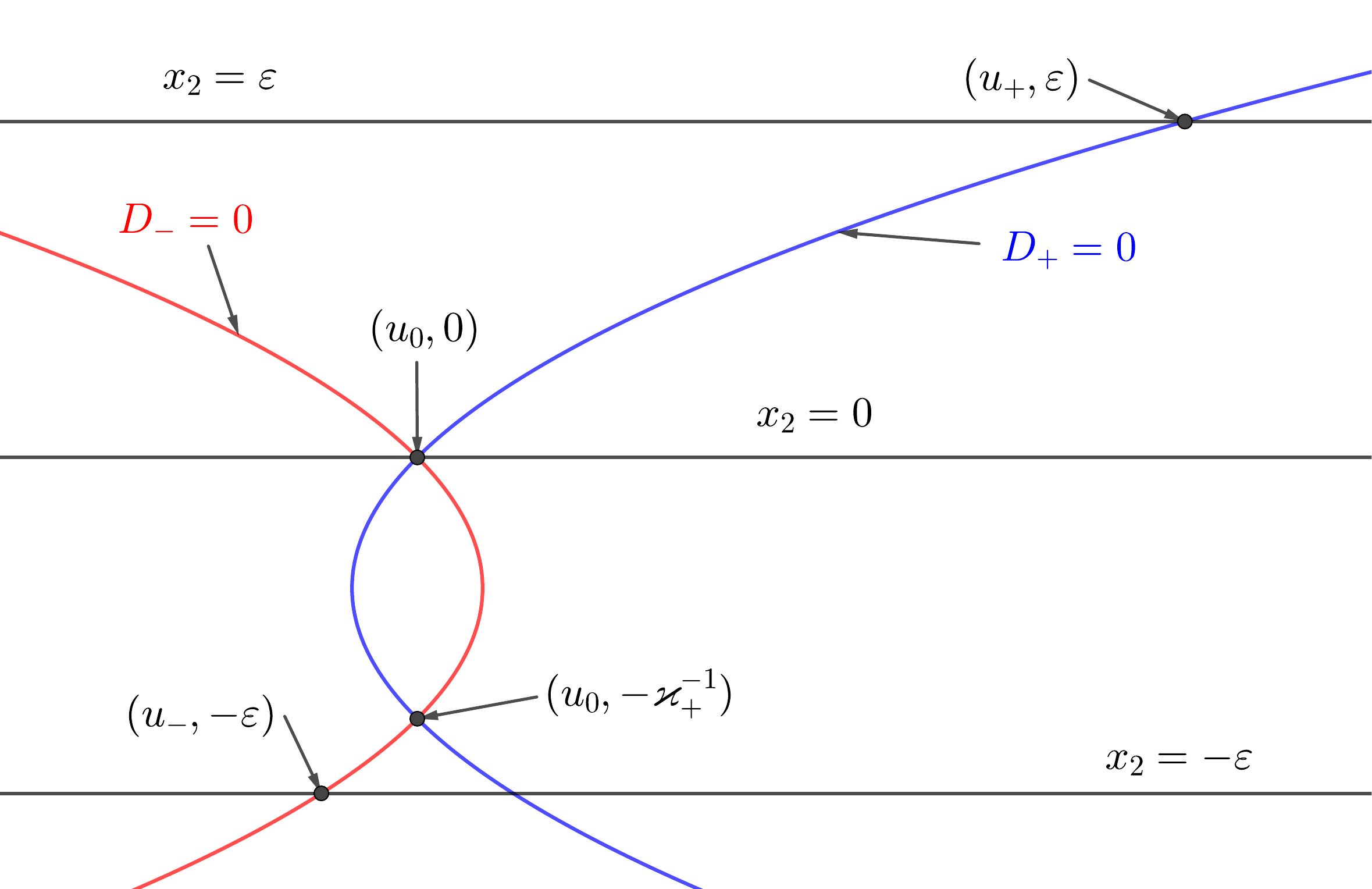}
    \caption{The curves $D_+=0$ and $D_-=0$. Four stationary points of the vector field.}
    \label{040301}
\end{figure}

To understand the behavior of integral curves around these points, we need to compute the Jacobian matrix at these points {\bf(}see~\eqref{011004}{\bf)}.
\eq{230203}{
J(x)=2(a_2^+-a_2^-)\begin{pmatrix}
\eps&
-3\vk_+x_2^2-2x_2
\\
-x_2&
2\vk_+\eps x_2-x_1+\eps+u_0
\end{pmatrix},
}
where {\bf(}see~\eqref{230704} and~\eqref{020301}{\bf)}
$$
\vk_+=\frac{6a_3}{a_2^+-a_2^-}\qquad\text{and}\qquad u_0=-\frac{a_1^+-a_1^-}{2(a_2^+-a_2^-)}\,.
$$

At $(u_0,0)$ we always have a node:
\eq{230204}{
J(u_0,0)=2(a_2^+-a_2^-)\eps\begin{pmatrix}
1&0
\\
0&1
\end{pmatrix}.
}
The point $(u_+,\eps)$ is a saddle for all $\eps$:
\eq{060301}{
J(u_+,\eps)=2(a_2^+-a_2^-)\eps\begin{pmatrix}
\;1&-2-3\eps\vk_+
\\
-1&\eps\vk_+
\end{pmatrix},
}
since $u_+=u_0+\eps+\vk_+\eps^2$ {\bf(}see~\eqref{070302}{\bf)}.
The behavior of the field at the other two stationary points is not important to our goal.

The simplest part of our task is to construct a simple right foliation in $\Om{R}(u_+,+\infty)$.
By direct calculations, we see that~\eqref{190501} turns into
\eq{270201}{
\begin{gathered}
2\eps D_+(u,\eps)=2(a_2^+-a_2^-)(u-\eps)+(a_1^+-a_1^-)-12a_3\eps^2\ge0\,;
\\
2\eps D_-(u,\eps)=2(a_2^+-a_2^-)(u+\eps)+(a_1^+-a_1^-)+12a_3\eps^2\ge0\,.
\end{gathered}
}
Note that the second inequality in~\eqref{270201} follows from the first one. 
The condition $D_+(u,\eps)>0$ holds for all $u>u_+$ because $D_+(u,\eps)$ is strictly increasing with respect to $u$, and $D_+(u_+,\eps)=0$.

Consider the integral curve of the field that starts at $(u_+,\eps)$, ends at $(v_-,-\eps)$ on the lower boundary, and corresponds to the spine of a SW-herringbone. Our reasoning will repeat the arguments from the proof of Proposition~\ref{200101}.
There, we used the fact that $\eps$ is small enough; here, we consider an arbitrary $\eps$, but we will use a specific expression for the boundary functions $f_\pm$.

To find the slope of our integral curve at $(u_+,\eps)$, we need to compute the eigenvectors of the matrix in~\eqref{060301}. Represent this matrix as
\eq{060302}{
\begin{pmatrix}
\;1&\ 1-3s
\\
-1&\!\!-1+s
\end{pmatrix},
}
where $s=1+\eps\vk_+$. The characteristic polynomial of this matrix has the form
\eq{060304}{
\lambda^2-s\lambda-2s=0\,,
}
thus the eigenvalues are $\lambda=\half(s\pm\sqrt{s^2+8s})$, and the eigenvectors are given by
\eq{190503}{
\begin{pmatrix}
s-1-\lambda
\\
1
\end{pmatrix}.
}

We choose one of the two integral curves passing through the point $(u_+,\eps)$, which, after a shift to the right by $\eps$, gives us the spine of the desired SW-herringbone. Specifically, we must take the eigenvector with a slope greater than $0$ and less than $1$. This means we take
$\lambda=\half(s-\sqrt{s^2+8s})$, which gives us the slope of the eigenvector equal to
\eq{070301}{
\frac1{s-1-\lambda}=\frac2{s-2+\sqrt{s^2+8s}}\,.
}
Compare this slope with the slope of the curve $D_+=0$. This curve is a parabola (see Fig.~\ref{040301}) given by the equation
\eqref{070302}.
Its slope is $x_2'=\frac1{1+2\vk_+x_2}$, so the slope at $(u_+,\eps)$ is
\eq{070303}{
\frac1{1+2\vk_+\eps}=\frac1{2s-1}\,,
}
which is strictly less than the expression in~\eqref{070301} for all $s>1$ (that is for all $\eps>0$).
Thus, in the upper half of the strip, the desired integral curve passes below the curve $D_+=0$ (or, equivalently, to the right of this parabola), i.\,e., in the domain where $D_+>0$. Since this curve is to the right of the parabola $D_-=0$, we have $D_->0$ on this integral curve.

The visualization of this picture is presented in Fig.~\ref{060303}, where in comparison with Fig.~\ref{040301}, we add several new lines and change the notation. The parabola $D_+=0$ (see~\eqref{070302}) is now denoted by $X_1$, meaning that the integral curve of our field intersects this line with a slope equal to 1. The same holds for the boundary $x_2=-\eps$ (except for the singular point), so it is also labeled as $X_1$. The parabola $D_-=0$ ($x_1=u_0-x_2-\vk_+x_2^2$) intersects the integral curves with a slope equal to $-1$ and therefore is denoted by $X_{-1}$. The same slope is on the upper boundary of the strip (except for the singular point). The set of points where integral curves pass horizontally is denoted by $X_0$. It consists of two lines: $x_2=0$ and $x_1=u_0+\eps+\vk_+\eps x_2$. The latter line passes through three stationary points: $(u_+,\eps)$, $(u_0,-\vk_+^{-1})$, and $(u_-,-\eps)$. The cubic parabola $X_\infty$ ($x_1=u_0+(x_2^2+\vk_+x_2^3)\eps^{-1}$) passes through all four stationary points and is the set of points where integral curves pass vertically (except for singular points). Finally, we add the integral curve that is our fissure. Starting from the point $(u_+,\eps)$ together with the parabola $X_1$, it passes below this parabola to the next stationary point $(u_0,0)$. The reasoning explaining why these two curves cannot intersect earlier and why the slope of the integral curve at the point $(u_0,0)$ is strictly less than 1 is contained in the proof of Proposition~\ref{200101}.

\begin{figure}[h]
    \centering
    \includegraphics[scale = 0.15]{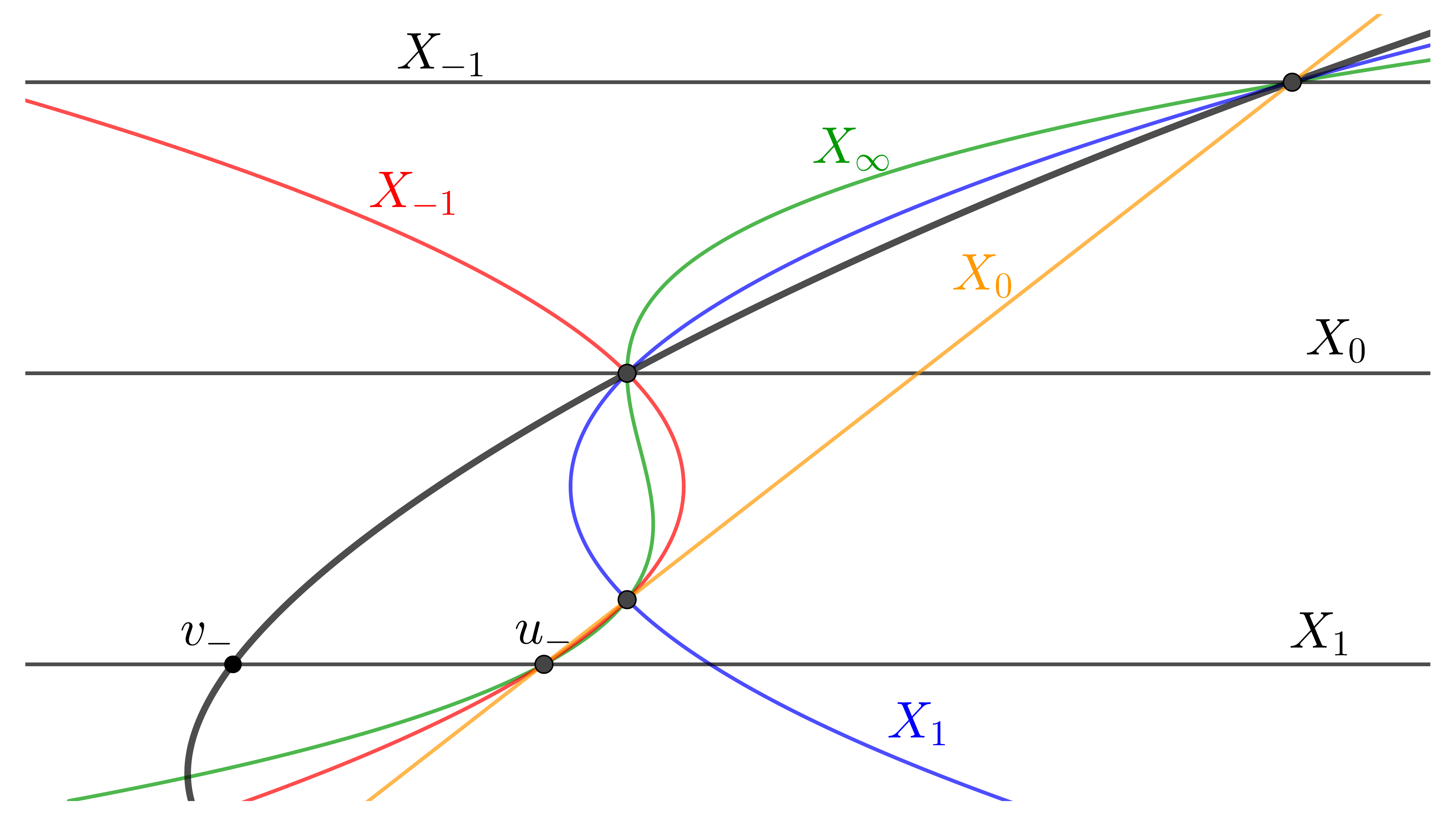}
    \caption{The integral curve of the vector field and the lines $X_0$, $X_{\pm1},$ and $X_\infty$.}
    \label{060303}
\end{figure}

We need to show that this integral curve in the lower half of the strip passes through the domain where $D_+>0$ and $D_->0$. To do this, we need to check that the end of the fissure $(v_-,-\eps)$ is to the left of the stationary point $(u_-,-\eps)$, i.\,e., $v_-<u_-$. Indeed, geometrically it is clear that to have a zero slope, our integral curve must turn: it has the slope equal to 1 at $(v_-,-\eps)$, then goes vertically when intersects $X_\infty$, then has the slope equal to $-1$  when intersects the parabola $X_{-1}$, and only then intersects the line $X_0$. All of this, of course, happens below the strip. This explains why $v_-<u_-$. This not only guarantees that this fissure generates a diagonally concave function but, as we will see shortly, it also ensures the possibility of constructing a simple left foliation in $\Om{L}(-\infty,v_-)$.

For the subsequent, we need to check  conditions~\eqref{200502} for $u\le v_-$. In our situation, they take the form
\eq{270202}{
\begin{gathered}
2\eps D_+(u,-\eps)=-2(a_2^+-a_2^-)(u+\eps)-(a_1^+-a_1^-)+12a_3\eps^2\ge0\,;
\\
2\eps D_-(u,-\eps)=-2(a_2^+-a_2^-)(u-\eps)-(a_1^+-a_1^-)-12a_3\eps^2\ge0\,.
\end{gathered}
}
Since the slope of our main curve is strictly less than 1, we have the inequality $v_-<u_0-\eps$, and the first condition in~\eqref{270202} follows. The second condition is precisely the inequality $v_-\le u_-$, which we have just verified.

Concluding the consideration of the case $a_3^+=a_3^-$, note that our SW-fissure is bounded for all~$\eps$, i.\,e., we always have $v_->-\infty$. Indeed, the description of the set $X_0$ shows that the only line $x_2=0$ can be an asymptote of the integral curve. Therefore, the spine of our herringbone must intersect the lower boundary of the strip at some finite point.

Finally, let us say a few words about the case $a_3^+\ne a_3^-$. Due to symmetry, we can consider the case $a_3^+>a_3^-$. As before, we need to investigate the roots of the polynomial
$$
f_+'(t)-f_-'(t)=3(a_3^+-a_3^-)t^2+2(a_2^+-a_2^-)t+(a_1^+-a_1^-)\,.
$$
If the discriminant of this quadratic polynomial is negative, then $f_+'(t)-f_-'(t)\ge\const>0$, and applying Proposition~\ref{061101}, we conclude that there is a simple right foliation for small $\eps$ in this case. However, as $\eps$ grows, foliations appear that we have not considered yet, so we postpone a complete study of this case.

We are also not ready to consider the foliation that arises when the discriminant of this quadratic polynomial is zero, that is, we have a multiple root of the equation~\eqref{011002}. Therefore, in the remaining part of this section, we will assume that the discriminant is positive, i.\,e., the equation~\eqref{210501} has two roots. Denote these roots as $u_{01}$ and $u_{02}$, assuming that $u_{01}<u_{02}$. The assumption $a_3^+>a_3^-$ means that $f_+''(u_{01})-f_-''(u_{01})<0$ and $f_+''(u_{02})-f_-''(u_{02})>0$. Therefore, we have three possible situations (see the illustration of these possibilities in Fig.~\ref{170301}--\ref{170303}):
\begin{itemize}
\item $a_3^->0$, and we have a NW-fissure near $u_{01}$ (by Proposition~\ref{200103}), $a_3^+>0$, and we have a~SW-fissure near $u_{02}$ (by Proposition~\ref{200101});
\item $a_3^-<0$, and we have a NE-fissure near $u_{01}$ (by Proposition~\ref{200104}), $a_3^+>0$, and we have a~SW-fissure near $u_{02}$ (by Proposition~\ref{200101});
\item $a_3^-<0$, and we have a NE-fissure near  $u_{01}$ (by Proposition~\ref{200104}), $a_3^+<0$, and we have a~SE-fissure near $u_{02}$ (by Proposition~\ref{200102}).
\end{itemize}
\begin{figure}[h]
    \centering
    \includegraphics[scale = 0.25]{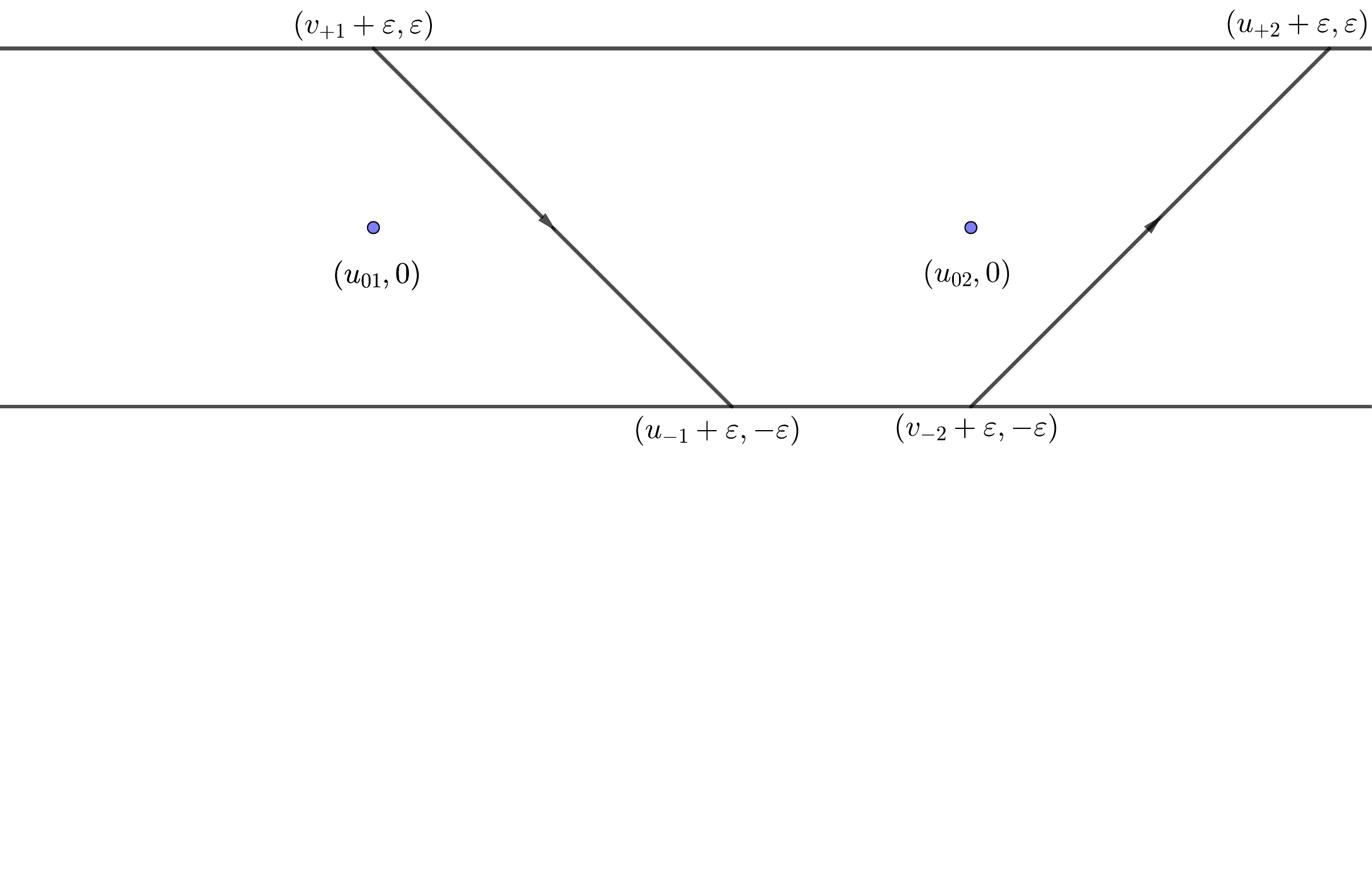}
		\vskip-90pt
    \caption{Two fissures if $a_3^->0$ and $a_3^+>0$.}
    \label{170301}
\end{figure}
\begin{figure}[h]
    \centering
    \includegraphics[scale = 0.25]{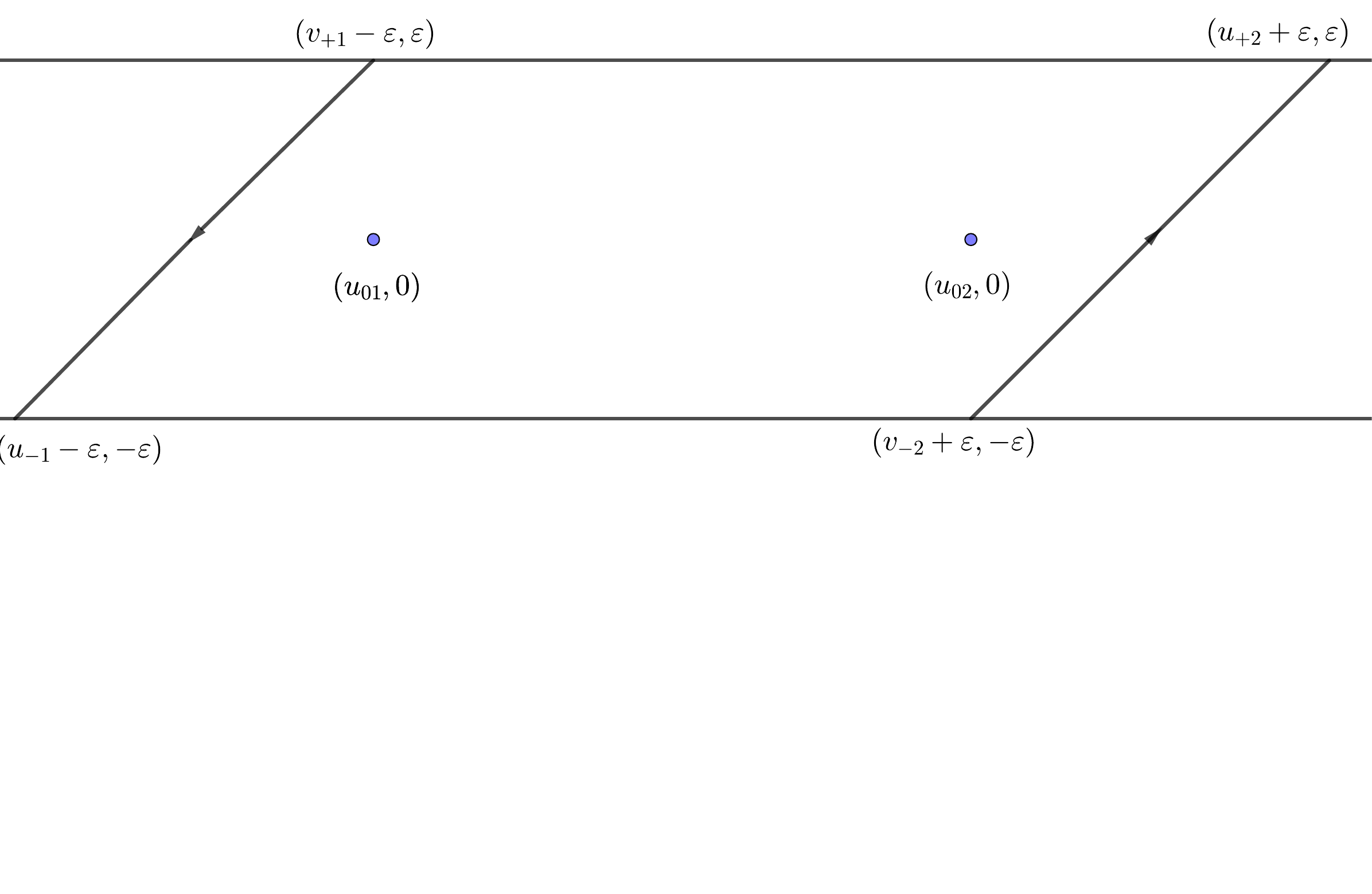}
		\vskip-90pt
    \caption{Two fissures if $a_3^-<0$ and $a_3^+>0$.}
    \label{170302}
\end{figure}
\begin{figure}[h]
    \centering
    \includegraphics[scale = 0.25]{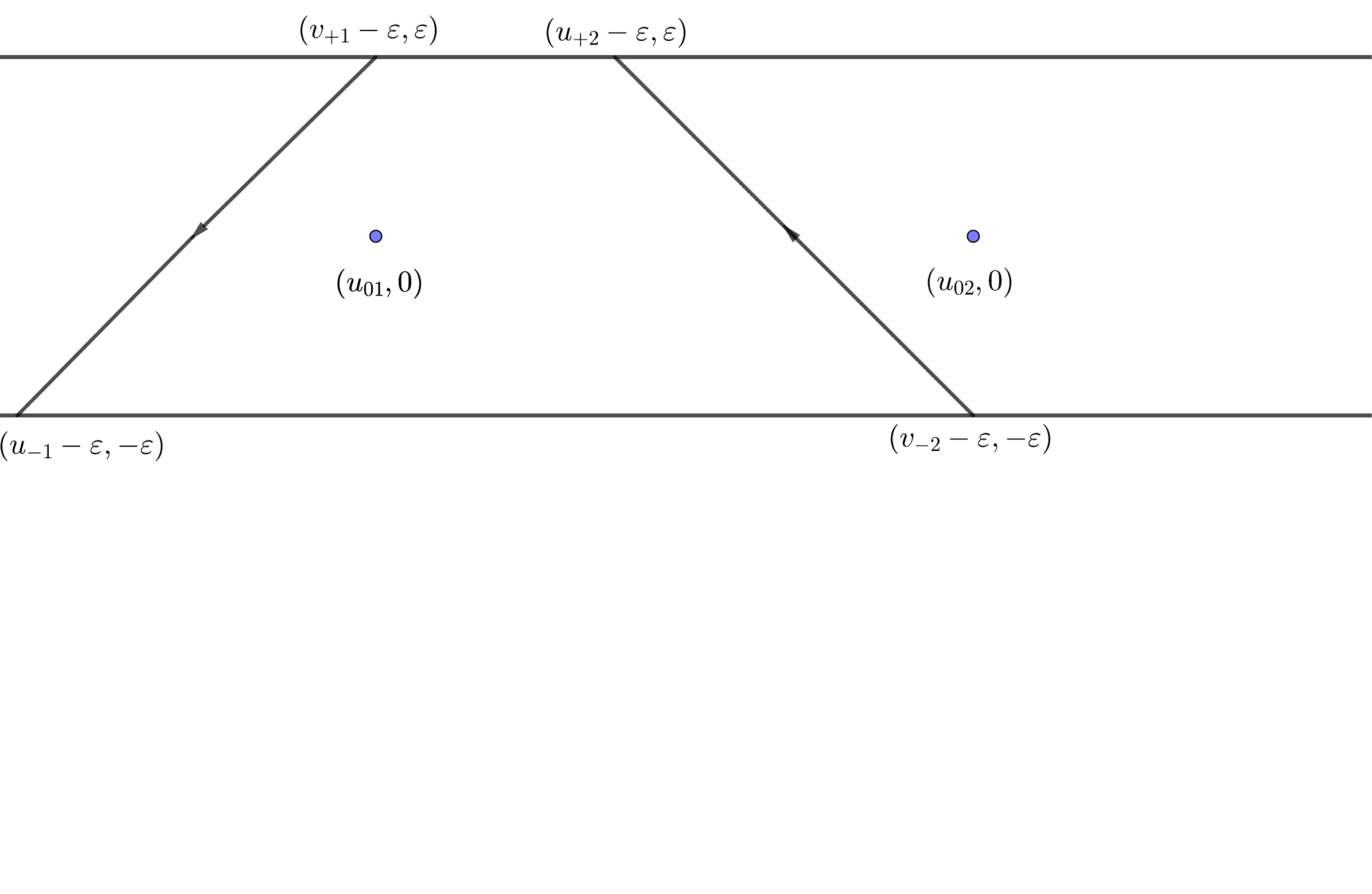}
		\vskip-90pt
    \caption{Two fissures if $a_3^-<0$ and $a_3^+<0$.}
    \label{170303}
\end{figure}
Note that the case $a_3^->0$, $a_3^+<0$ is impossible, because of the assumption $a_3^+>a_3^-$. 
Our aim is to prove that for small $\eps$ we always have the following foliation:
$$
\Oe=\Om{R}(-\infty,\alpha_1)\cup\Omega_\eps^{\scriptscriptstyle\mathrm{N}\ast}(\alpha_1,\beta_1)\cup\Om{L}(\beta_1,\alpha_2)
\cup\Omega_\eps^{\scriptscriptstyle\mathrm{S}\ast}(\alpha_2,\beta_2)\cup\Om{R}(\beta_2,+\infty)\,,
$$
where $\alpha_i$ and $\beta_i$ are some numbers close to $u_{i0}$, and instead of $\ast$ we have E or W depending on the signs of $a_3^\pm$. More precisely, we have to take
$$
\alpha_1=
\begin{cases}
v_{1+},\quad&\text{if }a_3^->0;
\\
u_{1-},\quad&\text{if }a_3^-<0,
\end{cases}
\qquad\beta_1=
\begin{cases}
u_{1-},\quad&\text{if }a_3^->0;
\\
v_{1+},\quad&\text{if }a_3^-<0,
\end{cases}
$$
and
$$
\alpha_2=
\begin{cases}
v_{2-},\quad&\text{if }a_3^+>0;
\\
u_{2+},\quad&\text{if }a_3^+<0,
\end{cases}
\qquad\beta_2=
\begin{cases}
u_{2+},\quad&\text{if }a_3^+>0;
\\
v_{2-},\quad&\text{if }a_3^+<0.
\end{cases}
$$

We need to verify the fulfillment of conditions~\eqref{190501} for $u<\alpha_1$ and $u>\beta_2$ (then we can construct $\Om{R}$), and conditions~\eqref{200502} on the interval $(\alpha_2,\beta_1)$ (then we can construct $\Om{L}$). We leave this verification to the reader, but here are several hints. All $u_\pm$ are precisely defined, and they determine the points at which the corresponding quadratic polynomials change sign. The second conditions in~\eqref{190501} and~\eqref{200502} are always weaker for small $\eps$ and are automatically satisfied. The boundaries defined by $v_\pm$ are not strict, and instead of checking the sign of the corresponding quadratic polynomials at these points, we check them at $u_{0i}$. For this purpose, it is sufficient to know that for all $\eps$, the following relations hold
$$
\begin{cases}
v_{+1}-u_{01}<-\eps,\quad&\text{if }a_3^->0;
\\
v_{+1}-u_{01}>\eps,\quad&\text{if }a_3^-<0,
\end{cases}
\qquad
\begin{cases}
v_{-2}-u_{02}<-\eps,\quad&\text{if }a_3^+>0;
\\
v_{-2}-u_{02}>\eps,\quad&\text{if }a_3^+<0,
\end{cases}
$$
because the slope of the spine lies strictly between $-1$ and $1$.

\section*{Aknowledgements}
The authors are grateful to D. Stolyarov, P. Ivanisvili, and A. Logunov, who participated in the discussion
on the questions under consideration at the initial stage.

\bigskip

V.~I.~Vasyunin

St. Petersburg State University, St. Petersburg, Russia,

vasyunin@pdmi.ras.ru

\bigskip

P.~B.~Zatitskii

University of Cincinnati, Cincinnati, OH, USA,

St. Petersburg State University, St. Petersburg, Russia,

pavelz@pdmi.ras.ru

\end{document}